\documentclass{amsart}

\usepackage{url}
\usepackage{amscd}
\usepackage{amsfonts}
\usepackage{amssymb}
\usepackage{euscript}
\usepackage{amsmath}
\usepackage{amsthm}
\usepackage[matrix, arrow, curve]{xy}
\usepackage{mathrsfs}

\usepackage{tikz}
\usepackage{tikz-cd}
\usetikzlibrary{arrows.meta}
\usepackage{dsfont}
\usepackage{enumitem}
\usepackage{derivative}
\usepackage {oplotsymbl}

\usepackage{appendix}

\newcommand{\Hol}[1]{\mathop{Hol(#1)}}

\newcommand{\M}[1]{\mathcal{M}_{0, #1}}
\newcommand{\Ms}[1]{\overline{\mathcal{M}}_{0, #1}}
\newcommand{\set}[1]{\underline{\mathbf{#1}}}
\newcommand{\PP}{\mathbb{P}^1}
\newcommand{\Gm}{\mathbb{G}_m}
\newcommand{\Gms}{\Gm\setminus\{1\}}
\newcommand{\EE}{\mathcal{E}}
\newcommand{\FF}{\mathcal{F}}
\newcommand{\LL}{\mathcal{L}}
\newcommand{\con}[2]{#1 \ast^! #2}

\newcommand{\jj}{\mathop{j_{\infty!}j_{0*}}}

\newcommand{\nea}[2]{\mathop{\psi_{#1}(#2)}}
\newcommand{\van}[2]{\mathop{\phi_{#1}(#2)}}
\newcommand{\Nea}[2]{\mathop{\Psi_{#1}(#2)}}
\newcommand{\Van}[2]{\mathop{\Phi_{#1}(#2)}}
\newcommand{\Shs}{\mathop{\mathop{\mathrm{Perv}}}(\Gm , 1)}
\newcommand{\Shsu}{\mathop{\mathop{\mathrm{Perv}^{un}}}(\Gm , 1)}

\newcommand{\Ho}[1]{\sideset{}{^{#1}}{\operatorname{\mathnormal{H}}}}
\newcommand{\Rde}{\operatorname{\mathbf{R}}}
\newcommand{\var}{var}
\newcommand{\can}{can}
\newcommand{\Var}{Var}
\newcommand{\ph}[1]{\varphi_{#1}}
\newcommand{\fg}[2]{\mathop{\pi_1}({#1}, {#2})}
\newcommand{\fgu}[2]{\mathop{\pi_1^{un}}({#1}, {#2})}
\newcommand{\pa}[3]{\mathop{\pi_1}({#1}, \allowbreak{#2}, \allowbreak{#3})}
\newcommand{\pau}[3]{\mathop{\pi_1^{un}}({#1}, {#2}, {#3})}
\newcommand{\Pau}[3]{\mathop{\Pi_1^{un}}({#1}, {#2}, {#3})}

\newcommand{\R}{\mathop{{\mathbf{R}}}}
\newcommand{\Rm}{\mathop{{\tau^{\le 0}\operatorname{\mathbf{R}}}}}

\newcommand{\kf}{\mathbf{k}}
\newcommand{\bra}[1]{\langle\!\langle #1 \rangle\!\rangle}
\newcommand{\cop}[1]{\mathop{\Delta_*}#1}
\newcommand{\intr}[1]{\mathop{\imath_{#1}}}
\newcommand{\aug}[1]{\mathop{\varepsilon}(#1)}
\newcommand{\pen}{\operatorname{\pentago}}
\newcommand{\Univer}[1]{\mathcal{U}_{#1}}
\newcommand{\ab}{\mathrm{ab}}
\newcommand{\muop}{\mathop{\mu}}
\newcommand{\PA}[2]{\mathop{\mathcal{P}}(#1,#2)}
\newcommand{\PAu}[2]{\mathop{\mathcal{P}^{un}}(#1,#2)}

\theoremstyle{plain}
\newtheorem{prop}{Proposition}
\newtheorem{theorem}{Theorem}

\newtheorem*{theorem*}{Theorem}
\newtheorem*{cor*}{Corollary}

\theoremstyle{definition}

\theoremstyle{remark}
\newtheorem{rem}{Remark}

\begin{document}

\title{Multiplicative convolution and double shuffle relations}
\author{Nikita Markarian}

\begin{abstract}
We develop a geometric approach to the regularized double shuffle relations for 
multiple zeta values, based on convolution of perverse sheaves on $\mathbb{C}^*$ 
and inspired by the approach of Deligne and Terasoma. We introduce semi-holonomy 
isomorphisms associated with pro-unipotent paths and show that their compatibility 
with multiplicative convolution is equivalent to a condition on the pro-unipotent 
fundamental group, the homological pentagon equation. We prove that this condition 
is equivalent to the regularized double shuffle relations, yielding a geometric 
proof that the pentagon equation implies these relations. The approach is purely 
topological and avoids Hodge-theoretic and Tannakian methods.
\end{abstract}

\email{nikita.markarian@gmail.com}

\date{}

\address{UMR 7501, Université de Strasbourg,
7 rue René Descartes,
67084 Strasbourg Cedex, France}
\maketitle

\section*{Introduction}

Multiple zeta values are periods given by iterated integrals on the interval, 
and hence appear as 
the coefficients of the Drinfeld associator \cite{Drin, LeM}. As such, they satisfy 
the relations imposed on the associator, namely the pentagon and hexagon equations 
introduced in \cite{Drin}. It was later shown in \cite{Fu} that the pentagon 
equation implies the hexagon equations. It is widely believed that the pentagon 
equation provides a complete set of geometric (or motivic) relations among multiple 
zeta values. The importance of understanding these motivic relations is underscored 
by the fact that multiple zeta values, viewed as motivic periods, generate the 
category of mixed Tate motives over $\mathbb{Z}$ \cite{Brown2012}.

Another fundamental family of relations is given by the regularized double shuffle relations, introduced in \cite{IKZ, Racinet}.
In \cite{IKZ}, it was conjectured that the regularized double shuffle relations form a complete set of relations among multiple zeta values.
Clarifying the precise relationship between associator relations and regularized double shuffle relations, and understanding whether one family implies the other or whether intermediate structures exist, remains a central open problem in the theory of multiple zeta values.

The regularized double shuffle relations
arise from rearrangements of terms in the product of series defining multiple zeta 
values, although they also admit a geometric formulation (see, for example, 
\cite{Mar_cell}). Such manipulations of series can be interpreted in terms of the 
multiplication of Mellin transforms of functions given by polylogarithms supported 
on the interval $[0,1]$, which corresponds to convolution of polylogarithms (see \cite{15}).

This observation can be viewed as a motivation for the program of Deligne and 
Terasoma \cite{DLet, DTerICM, DTerPr}, which aims to relate double shuffle 
relations to convolution of perverse sheaves. A central goal of this approach is to 
show that the associator relations imply the regularized double shuffle relations. 
Although this implication has since been established by other methods, the geometric perspective remains of 
considerable interest.

In this paper, I develop the approach of Deligne and Terasoma.
There are two main differences from their approach.
First, I do not use Hodge theory; all constructions are purely topological.
The Hodge-theoretic aspects of this approach are considered in 
\cite{17}, where I study the corresponding mixed Hodge structures and their failure to be preserved by the semi-holonomy isomorphism.
Second, instead of working within a Tannakian formalism for a Serre quotient of the category of perverse sheaves, I give a more explicit treatment, avoiding this formalism altogether.

One of the main results of this paper is a new geometric proof that the pentagon equation implies the regularized double shuffle relations. More generally, I introduce the homological pentagon equation and prove that it is equivalent to the regularized double shuffle relations.
I also establish several auxiliary results that may be of independent interest.

The statement that the pentagon equation implies the regularized double shuffle relations was proved in \cite{Furusho2011, HiroseSato, Furusho2022}.
The approach of Deligne–Terasoma has been further developed in \cite{Enriquez2021, Enriquez2022, Enriquez2023}.
The Betti realization considered there is close in spirit to the approach of the present paper, but remains purely algebraic in nature.

A central feature of this work is the development of a purely geometric approach to the regularized double shuffle relations.
The framework used to describe multiplicative convolution of sheaves is closely related to the geometric construction of the self-intersection map, a construction sketched in the Appendix and inspired by \cite{HainTur}.
This suggests a connection between the self-intersection map and the harmonic coproduct, which will be explored in future work.

More broadly, the present framework suggests a possible geometric explanation of the relationship between the Kashiwara--Vergne problem and double shuffle relations
(see, e.g., \cite{SCHNEPS_KV}).
%

The main tool used in this paper is what I call semi-holonomy.
Semi-holonomy provides anothe incarnation of the associator:
 different choices of paths produce different isomorphisms, and certain relations among these isomorphisms encode the regularized double shuffle relations.
It is an isomorphism between the vanishing cycles at $1$ of a perverse sheaf $\FF$ on $\mathbb{C}^*$, which is smooth and pro-unipotent outside $1$, and the space $\Ho{0}(\PP,\, \jj \FF)$, where $\jj$ denotes the extension of $\FF$ by $!$ at $\infty$ and by $*$ at $0$.
An isomorphism of this type already appeared in a different context in \cite{Katz+2012}.

The key observation is that there are many such isomorphisms, depending on the choice of a pro-unipotent path from the tangential base point at $1$ to that at $0$ in $\mathbb{C}^*\setminus{1}$.
I call the resulting map the semi-holonomy along this path.
The standard choice of the isomorphism corresponds to the straight interval $[0,1]$.
This isomorphism is transcendental and does not preserve the mixed Hodge structures carried by both sides, although it preserves the weight structure.
This point of view provides a new perspective on the Drinfeld associator
(see \cite{16}).

A crucial feature of the Deligne--Terasoma approach is that vanishing cycles define a tensor functor with respect to multiplicative convolution of sheaves.
This is analogous to the Thom--Sebastiani theorem \cite{Sebastiani1971, Massey2001}, which identifies vanishing cycles of a sum of singularities with the tensor product of the vanishing cycles of the summands.
One may ask whether the isomorphisms provided by semi-holonomy are compatible with this tensor structure.
This is indeed the case for the semi-holonomy along the interval, and this constitutes the first main result of the paper.

This naturally leads to the problem of characterizing those pro-unipotent paths for which semi-holonomy is compatible with convolution.
The resulting condition, which I call the homological pentagon equation, turns out to be equivalent to the regularized double shuffle relations for the corresponding element of the completed free algebra.
In particular, the pentagon equation implies this compatibility.

Another important ingredient in the Deligne--Terasoma approach is the transport algebra $W$, which acts on vanishing cycles.
I  describe explicitly the difference between semi-holonomies in terms of $W$. 
The resulting formula, expressed via the transport algebra and Fox derivatives, is of independent interest.

It is shown in \cite{DTerPr} that the category of modules over $W$ can be realized as a certain Serre quotient of the category of perverse sheaves on $\mathbb{C}^*$ that are smooth and pro-unipotent outside $1$.
The multiplicative convolution endows this category with a symmetric monoidal structure, which, via Tannakian formalism, gives rise to a coproduct on $W$, the harmonic coproduct, which is another central notion of the Deligne--Terasoma approach.
In this paper, I give a direct geometric construction of this coproduct.
In particular, I avoid passing to quotient categories by introducing the notion of a sheaf without sections supported at $1$.
This subcategory is stable under multiplicative convolution.

I also note a connection with the results of \cite{AHNRS}. These results, which rely substantially on the theory developed in \cite{HR} and its antecedents in \cite{Furusho2011}, exhibit features that are closely related to those appearing in the present work. 

The remainder of the paper is organized as follows.
In Section 1, I introduce the semi-holonomy along the interval.
Section 2 introduces notation and recalls several constructions used throughout the paper.
In Section 3, I construct certain nearby and vanishing cycles associated with convolution.
While these constructions admit a simple geometric interpretation, I present them directly from the definitions, with a view toward future applications to Hodge theory.
In Section 4, I prove the compatibility of the standard semi-holonomy isomorphism with multiplicative convolution.
In Section 5, I introduce the transport algebra and define the harmonic coproduct on it.
In Section 6, I give an explicit formula describing semi-holonomy in terms of the action of the transport algebra and Fox derivatives.
In Section 7, I introduce the notion of the homological pentagon equation and prove its equivalence with the regularized double shuffle relations, building on the results of the preceding sections. In the appendix I give a homological definition of the self-intersection map.

\smallskip
{\bf Acknowledgments.} I am grateful to A.~Alekseev, B.~Enriquez, M.~Finkelberg, H.~Furusho and M.~Kapranov for fruitful and insightful discussions.
I would like to thank the Max Planck Institute for Mathematics and IHES for hospitality and excellent
        work conditions. 
This project is supported by the PAUSE program and the ITI IRMIA++.

\subsection*{Notations}

All algebraic varieties are defined over $\mathbb{C}$. All sheaves take values in the category of vector spaces over a field $\kf$ of characteristic $0$.

$\PP$ is the projective line, $\Gm=\PP\setminus\{0, \infty\}$ is the multiplicative group.

$\Shs$ (resp. $\Shsu$) denotes the category of perverse sheaves on $\Gm$ that are smooth (resp. smooth and pro-unipotent) outside $1$.

We omit the derived functor notation: for instance, $j_*$ stands for $\Rde j_*$ and $i^!$ for $\Rde i^!$.

Let $f\colon X \to D \subset \mathbb{C}$ be a morphism from a complex algebraic variety $X$ to a disc $D \subset \mathbb{C}$,
with the central fiber  $X_0 = f^{-1}(0)$, the inclusion map $i\colon X_0 \to X$,
and let $\FF^\bullet$ be a complex of constructible sheaves on $X$. 
We denote by $\phi_f(\FF^\bullet)$ and $\psi_f(\FF^\bullet)$ the complexes of vanishing and nearby cycles on $X_0$.

The functors $\van{f}{-}[-1]$ and $\nea{f}{-}[-1]$ preserve the category of perverse sheaves.
They fit into the (Milnor) distinguished triangles:
\begin{equation}
\begin{tikzcd}[cramped]
\nea{f}{\FF^\bullet}[-1]  \arrow[r, "can"]
& \van{f}{\FF^\bullet}[-1]  \arrow[r] 
& i^*\FF^\bullet \arrow[r, "+1"] 
& {}
\end{tikzcd}
\label{can}
\end{equation}

\begin{equation}
\begin{tikzcd}[cramped]
i^! \FF^\bullet \arrow[r]
& \van{f}{\FF^\bullet}[-1]  \arrow[r,"var"]
& \nea{f}{\FF^\bullet}[-1] \arrow[r, "+1"] 
& {}
\end{tikzcd}
\label{var}
\end{equation}

We denote the cohomology of vanishing and nearby cycles at a point by capital letters:
\begin{equation}
\Van{f}{\FF}=\Ho{0}(\van{f}{\FF^\bullet}[-1]) \quad \mbox{and} \quad \Nea{f}{\FF}=\Ho{0}(\nea{f}{\FF^\bullet}[-1] )
\end{equation}

For $a \in {0,1,\infty}$, we denote by $i_a\colon \mathrm{pt} \to \PP$ the inclusion of the point $a$.

We denote by $1-z$ and $z$ the tangential base points of $\M{4}$ at $1$ and $0$, respectively, defined by the corresponding functions.
For brevity, we also denote these tangential base points simply by $1$ and $0$.

\section{Deligne fiber functor}
\label{deligne}

\subsection{Vanishing cycles}
\label{van1}

Firstly, prove a general useful result.
Let $I$ be the real interval $[0, 1]$ in $\PP$ not containing $\infty$, let
$i_I\colon I\to \PP$ be the closed embedding.
Denote by $j_\infty$ the inclusion of $\mathbb{A}^1$ into $\PP$.
 
\begin{prop}[Shrinking]
For a complex of constructible sheaves $\FF^\bullet$ on  $\mathbb{A}^1$ smooth
outside interval $I$,  composition of the excision isomorphism and
the   forgetful map
give an isomorphism
\begin{equation}
\begin{tikzcd}[cramped, sep=small]
H^*(I, \,i_I^! \FF^\bullet)\arrow[r,equal]
& H^*_I(\PP,\, j_{\infty !} \FF^\bullet) \arrow[r]
& \Ho{*}(\PP,\, j_{\infty!}\FF^\bullet)
\end{tikzcd}
\end{equation}
\label{basic}
\end{prop}
\begin{proof}
Let 
$j_{\setminus I}\colon \PP\setminus I\to \PP$ be the open embedding. Consider the exact triangle
\begin{equation}
\begin{tikzcd}[cramped, sep=small]
{i_I}_*i_I^! \FF^\bullet \arrow[r, equal]
&{i_I}_*i_I^! j_{\infty!}\FF^\bullet \arrow[r]
& j_{\infty!}\FF^\bullet \arrow[r] 
&  j_{\setminus I *} {j_{\setminus I}}^*j_{\infty!} \FF^\bullet \arrow[r, "+1"] 
& {}
\end{tikzcd}
\label{retract}
\end{equation}
The pair of topological space $(\PP\setminus I,\, \PP\setminus \{I, \infty\}) $ is homeomorphic to the pair $(\mathbb{A}^1,\, \Gm)$.
The restriction of  $\FF^\bullet$  
on $\PP\setminus \{I, \infty\}$ is a complex of local systems.
To show that cohomology of its $!$-extension on the bigger space
vanishes, consider the corresponding complex of local systems on $\Gm$.
One may see that for a 
local system $\LL$ on  $\Gm$, cohomology  $\Ho{*}(\mathbb{A}^1,\, j_{\infty!} \LL)$ vanishes, which follows that cohomology of a complex of local systems vanishes as well.
It follows that cohomology of the last term in (\ref{retract})
vanishes.
Thus, cohomologies of first two terms in (\ref{retract})
are isomorphic.
\end{proof}

Let $j\colon \Gm \to \PP$ be the embedding of the multiplicative group to 
the projective line.
Factor $j$ as $j_\infty \circ j_0$, where $j_0$ is the
inclusion of $\Gm$ into $\mathbb{A}^1$, and $j_\infty$ is the inclusion of $\mathbb{A}^1$ into $\PP$.

Let $\FF$ be a perverse sheaf from $\Shs$.
Denote by $(0, 1]\subset \Gm$ the intersection of $I$ and $\Gm$.
Combining the canonical isomorphism of local cohomologies
 $$H^\bullet_{(0,1]}(\Gm,\, \FF) \simeq H^\bullet_{(0,1]}(\mathbb{A}^1,\, j_{0*}\FF) $$
with the excision isomorphism 
 $$ H^\bullet_{(0,1]}(\mathbb{A}^1,\, j_{0*}\FF)\simeq H^\bullet_{(0,1]}(\PP,\, \jj\FF) $$
 followed by 
 the  forgetful map   to $H^\bullet(\PP,\, \jj\FF)$ gives a map from
$H^0_{(0,1]}(\Gm,\, \FF)$ to  $H^0(\PP,\, \jj\FF)$.

\begin{prop}
For a  perverse sheaf $\FF\in \Shs$  
cohomology $H^0_{(0,1]}(\Gm,\, \FF)$ is naturally isomorphic to vanishing cycles $\Van{{1-z}}{\FF} $ of $\FF $ at $1$ 
and 
the  map defined above gives an isomorphism
\begin{equation}
\begin{tikzcd}[cramped, sep=small] \ph{I} \colon \Van{{1-z}}{\FF} \arrow[r] &  \Ho{0}(\PP,\, \jj \FF) \end{tikzcd}
\label{equality}
\end{equation} 
 For $i\ne 0$ cohomologies $\Ho{i}(\PP,\, \jj \FF)$ vanish.
\label{phi}
\end{prop}

\begin{proof}
Apply Proposition \ref{basic} to the perverse sheaf $j_{0!}\FF$.
As $H^0_{(0,1]}(\Gm,\, \FF)$ is equal to $H^0_I(\mathbb{A}^1,\, j_{0!}\FF)$, it gives  isomorphism (\ref{equality}). An isomorphism 
between $H^0_{(0,1]}(\Gm,\, \FF)$ and $\Van{{1-z}}{\FF}$
is established in \cite{Galligo1985} and may be taken as a definition of vanishing cycles. Moreover, it is shown there, that
 the exact triangle
\begin{equation}
\begin{tikzcd}[cramped]
H^\bullet(i_1^! \FF)\arrow[r]
& H^\bullet_{(0, 1]}(\Gm,\,  \FF) \arrow[r] 
& H^\bullet_{(0, 1)}(\Gm,\, \FF) \arrow[r, "+1"]
& {}
\end{tikzcd}
\label{complex}
\end{equation}
associated with the pair $(0, 1)\subset (0, 1]$ in $\Gm$   is isomorphic to the cohomology of the Milnor triangle (\ref{var}) for $\FF$, see also
 \cite[Remark 9.5]{KapSch}.
\end{proof}

Above we constructed not only the class in $\Ho{1}(\PP,\, \jj \FF[-1])$,
but the corresponding extension. This extension obeys the
natural  condition, which essentially specify in which direction 
interval $I$ meets point $1$, see Proposition \ref{frame} below.

\begin{rem}
Following \cite[Section 3]{Katz+1991}, one may show that the dimension of
$\Van{{1-z}}{\FF}$ coincides with the one of $\Ho{0}(\PP,\, \jj \FF)$ by means of 
the exactness properties of push-forwards along affine morphisms of perverse sheaves, see
\cite[4.1]{BBD}. Indeed, they imply that 
$\Ho{>0}(\PP,\, \jj \FF)=\Ho{>0}(\PP\setminus\{0\},\, j_!^\infty\FF)=0$
and also $\Ho{<0}(\PP,\, \jj \FF)=0$, by the Verdier duality. Thus, 
dimension of  $\Ho{0}(\PP,\, \allowbreak \jj \FF)$ equals to the Euler characteristic
of the cohomology, which can be shown to be equal to the dimension of $\Van{{1-z}}{\FF}$. This gives an alternative proof of equality of dimensions, but not the  isomorphism constructed above.
\label{purite} 
\end{rem}

\subsection{Nearby cycles}
For  a perverse sheaf on $\PP$,
which is an extension of 
a shifted local system from  $\PP\setminus \{0, \infty, 1\}$
by $*$, $*$, and $!$ in some order, Proposition \ref{phi} implies an isomorphism
between vanishing cycles at points, to which the local system is extended by $*$.
These vanishing cycles are isomorphic to nearby cycles,
the isomorphism is given by $var$ from (\ref{var}).
The following proposition reveals the nature of this isomorphism.

Denote by $j_1\colon \Gm\setminus\{1\} \to \Gm$ the open embedding
and by 
\begin{equation}
r\colon z\mapsto 1-z 
\label{reflexion}
\end{equation}
the automorphism of $\PP$.

\begin{prop}
Let $\LL$ be a local system on $\Gm\setminus \{1\}$.
The composition of isomorphisms
\begin{equation}
\begin{tikzcd}[cramped]
\Nea{{1-z}}{j_{1*}\LL[1]}
& \Van{{1-z}}{j_{1*}\LL[1]} \arrow[r, "\ph{I}"] \arrow[l, "var"']
& \Ho{0}(\PP,\, \mathop{j_{\infty!}j_{0*} j_{1*}}\LL[1])\arrow[d,"r"] \\
 \Nea{{z}}{j_{1*}\LL[1]}
& \Van{{z}}{j_{1*}\LL[1]}  \arrow[l, "var"'] \arrow[r, "\ph{I}"]
& \Ho{0}(\PP,\, \mathop{j_{\infty!}j_{1*} j_{0*}}\LL[1])
\end{tikzcd}
\end{equation}
is the holonomy along $I$.
\label{hol}
\end{prop}

\begin{proof}
The statement follows from the proof of Proposition~\ref{phi}. For $j_{1*}\LL[1]$, the isomorphism $\ph{I}$ identifies the space with sections of $i_I^!\LL[1]$ over the interior of $I$. The identification with nearby cycles is obtained by taking limits of these sections as one approaches the endpoints.
\end{proof}

\begin{rem}
Thus, $\Ho{0}(\PP,, j_{\infty!}j_{0*} j_{1*}\LL[1])$ may be thought of as the space of sections of $\LL$ over the interval $(0,1)$, and the isomorphisms $\ph{I}$ as the specialization maps to the nearby cycles at the endpoints. Compare~\cite[Introduction]{DGal}.
\end{rem}

Morphisms between different extensions of a smooth perverse sheaf 
on $\Gm\setminus \{1\}$ may be expressed in terms of Milnor triangles.

\begin{prop}
The following diagram commutes
\begin{equation}
\begin{tikzcd}[cramped]
\Ho{0}(i_1^!\FF) \arrow[r] \arrow[d, equal]
& \Ho{0}(\PP,\, \jj \FF) \arrow[r] 
& \Ho{0}(\PP,\, \mathop{j_{\infty!}j_{0*} j_{1*}j_1^*}\FF)\\
\Ho{0}(i_1^!\FF)  \arrow[r]
& \Van{{1-z}}{\FF} \arrow[r, "var"] \arrow[u,"\ph{I}"]
& \Nea{{1-z}}{\FF}=\Van{{1-z}}{j_{1*}j_1^*\FF} \arrow[u,"\ph{I}"] 
\end{tikzcd}
\end{equation}
where the second line is the cohomology of the Milnor triangle (\ref{var})
and the first line is the cohomology of the standard trianlge
\begin{equation}
\begin{tikzcd}[cramped]
{i_1}_*i_1^! \FF\arrow[r]
& \jj\FF \arrow[r] 
& \mathop{j_{1*}j_1^*j_{\infty!}j_{0*} }\FF \arrow[r, "+1"]
& {}
\end{tikzcd}
\end{equation}
\label{nea2}
\end{prop}

\begin{proof}
The statement follows immediately from the proof of Proposition \ref{phi}.
\end{proof}

The  statements analogous to two propositions above may be formulated for the $!$-extension as well.

Prove a technical proposition we will need below.
An element $v\in \Van{{1-z}}{\FF} $ gives an extension, which represents the class
$\Ho{1}(\PP,\allowbreak\, \jj \FF[-1])$, given by Proposition \ref{phi}.
Combining this extension with the connecting morphism in the standard triangle
\begin{equation}
\begin{tikzcd}[cramped]
\mathop{i_{0*}i_0^! j_{0!}}\FF\arrow[r]
& \mathop{j_{0!} j_{\infty!} }\FF \arrow[r] 
& \mathop{j_{0*} j_{\infty!} }\FF \arrow[r, "+1"]
& {}
\end{tikzcd}
\label{zero}
\end{equation}
we get an extension representing $\Ho{1}(i_0^! j_{0!}\FF)$.

On the other hand, applying the isomorphism of Proposition \ref{hol} to \var(v) we get an element of $\Nea{{z}}{j_{0!}\FF}$ at point $0$.
With such an element, one may associate an extension  representing $\Ho{1}(i_0^! j_{0!}\FF)$ by means of the Milnor triangle      (\ref{var}).

\begin{prop}
Two extensions  introduced above, representing class in 
$\Ho{1}(i_0^! j_{0!}\FF)$, coincide.
\label{frame}
\end{prop}

\begin{proof}
The statement follows immediately from the standard definition of nearby cycles
and the Milnor triangle.
\end{proof}

\begin{rem}
Verdier duality gives isomorphisms between $\Ho{0}(\PP,\, \mathop{j^a_!j^b_* j^c_*}\LL[1])$ and $\Ho{0}(\PP,\, \mathop{j^a_*j^b_! j^c_!}\LL[1])$
for $\{a,b,c\}=\{0, \infty, 1\}$. Combining it with  natural maps
from $!$-extensions to $*$-extensions as in the proposition above,
one obtains the action on the fiber of $\LL$ of monodromies at all infinite 
points, that is the action of the fundamental group of $\Gm\setminus\{1\}$.
\end{rem}

\section{Multiplicative convolution}
\label{MConv}

\subsection{Moduli spaces}
\label{moduli}

Denote by $\M{n}$ the moduli space of embeddings of an $n$-element set $\set{n}\hookrightarrow \mathbb{P}^1$, considered up to the action of the Möbius group. This is a smooth affine variety admitting a smooth projective compactification $\Ms{n}$, the moduli space of stable curves. The complement $\Ms{n}\setminus\M{n}$ is a union of normal crossing divisors, indexed by partitions of $\set{n}$ into two subsets of cardinality at least $2$.
Denote by $p_i\colon \M{n} \to \M{n-1}$ and $p_i\colon \Ms{n} \to \Ms{n-1}$ the forgetful maps that omit the $i$-th point.

Choose coordinates on $\PP$ so that the first three elements of $\set{n}$ are placed at $0$, $\infty$, and $1$, respectively. The coordinates of the remaining points are called \emph{simplicial coordinates} on $\M{n}$.
The coordinates of the remaining points define simplicial coordinates on $\M{n}$. Thus, a point of $\M{n+3}$ with simplicial coordinates $(t_1,\dots,t_n)$ is given by $(0,\infty,1,t_1,\dots,t_n)$.

In these coordinates, $\M{4}$ is identified with $\PP\setminus{0,\infty,1}$, and $\M{5}$ with $\M{4}\times\M{4}\setminus \Delta$, where $\Delta$ is the diagonal.
The projections from $\M{5}$ to $\M{4}$ look
in simplicial coordinates as follows:
\begin{equation}
\begin{aligned}
p_3(0, \infty, 1, t_1, t_2)&=(0, \infty, 1, t_2/t_1)\\
p_4(0, \infty, 1, t_1, t_2)&=(0, \infty, 1, t_2)\\
p_5(0, \infty, 1, t_1, t_2)&=(0, \infty, 1, t_1)
\end{aligned}
\label{e5}
\end{equation}

Space $\Ms{4}$ is the obvious compactification of $\M{4}$, 
which is $\PP$. $\Ms{5}$ is isomorphic to 
$\Ms{4}\times\Ms{4}$ with blown-up points $(0,0)$, $(1,1)$ and $(\infty, \infty)$.

Denote by $\Ms{5}'$ the space obtained from $\Ms{4}\times\Ms{4}$ by blowing up the points $(0,0)$ and $(\infty,\infty)$. Equivalently, $\Ms{5}'$ is obtained from $\Ms{5}$ by blowing down the divisor corresponding to stable curves of type $(12)(345)$.
Denote by
\begin{equation}
p'_i \colon \Ms{5}' \to \Ms{4} 
\label{mp}
\end{equation}
the forgetful maps, as in (\ref{e5}).

\subsection{Convolution}
\label{Convolution}

Let $m\colon \Gm\times\Gm\to \Gm$ be the multiplication map on the multiplicative group.
The convolution of two perverse sheaves $\EE$ and $\FF$ on $\Gm$ is
a perverse sheaf on $\Gm$
defined by
\begin{equation}
\con{\EE}{\FF} = \Rm m_!(\EE\boxtimes\FF)
\label{box}
\end{equation}
Since the map $m$ is affine, the functor $m_!$ is $t$-exact for the perverse $t$-structure, hence the convolution of perverse sheaves is again perverse (see \cite[4.1]{BBD}).

\begin{rem}
In general, $\mathbf{R}m_!(\EE\boxtimes\FF)$ is a complex of perverse sheaves rather than a single perverse sheaf, which explains the appearance of the truncation $\tau^{\le 0}$ in the definition. This issue can be avoided by working in an appropriate quotient category, as in \cite{DTerICM, DTerPr, DLet} (see Theorem~\ref{Quotient}), since the higher cohomology sheaves are smooth at $1$ (cf. Proposition~\ref{conv2}).
\end{rem}

Let $\imath\colon \PP \to \PP$
be the inversion map $z\mapsto z^{-1}$.
For a sheaf $\FF$ on $\PP$ introduce the notation
\begin{equation*}
\overline{\FF}= \imath^*\FF
\end{equation*}

Recall, that in Section \ref{deligne} we denoted by
$\jj$ the $*$-extension at $0$ and $!$-extension at $\infty$ of a perverse sheaf from 
$\Gm$ to $\PP$. 

\begin{prop}[\cite{DLet, Katz+2012}]
For perverse sheaves $\EE$ and $\FF$ on $\Gm$
we have
\begin{equation}
\jj (\con{\EE}{\FF})=\Rm p'_{3*}({p'_5}^*(\overline{\jj\EE}) \otimes {p'_4}^*(\jj\FF))
\label{ext}
\end{equation}
where $p'_i$ are defined in (\ref{mp}).
\label{conv1}
\end{prop}

\begin{proof}
The proof proceeds in two steps. 

First, we construct an isomorphism
\begin{equation}
\con{\EE}{\FF} \cong \Rm p'_{3*}({p'_5}^*(\overline{\jj\EE}) \otimes {p'_4}^*(\jj\FF))
\label{long}
\end{equation}
on $\Gms$.
By (\ref{e5}), on the fibers of $p'_3$ the sheaves
${p'_5}^*(\overline{\jj\EE}) \otimes {p'_4}^*(\jj\FF)$ and $\EE\boxtimes\FF$
are isomorphic away from $0$ and $\infty$. The former is the $!$-extension of the latter across $0$ and $\infty$, hence its $*$-pushforward coincides with the $!$-pushforward of $\EE\boxtimes\FF$.

It remains to show that $\Rm p'_{3*}(\cdots)$ is the $*$-extension at $0$ and the $!$-extension at $\infty$ of $\con{\EE}{\FF}$.
Consider the fiber
of $p'_3$ over $0$. It consists of two intersecting lines,
and ${p'_5}^*(\overline{\jj\EE}) \otimes {p'_4}^*(\jj\FF)$
is $*$-extension of $\EE\boxtimes\FF$ to these lines.
Thus, it is true for the pushforward as well. 
Over $\infty$ the situation is the same.
\end{proof}

\begin{prop}[\cite{DLet, Katz+2012}]
For perverse sheaves $\EE, \FF\in \Shs$, 
\begin{equation}
\Ho{0}(\PP,\, \jj(\con{\EE}{\FF}))=\Ho{0}(\PP,\, \jj\EE)\otimes \Ho{0}(\PP,\,\jj\FF)
\label{otimes}
\end{equation}
\label{conv2}
\end{prop}

\begin{proof}
Consider the Leray spectral sequence for $\R p'_{3*}(\cdots)$ from \eqref{long}. One checks that the higher direct images are smooth on $\Gm \setminus 1$ and extend as $*$-extensions at $0$ and $!$-extensions at $\infty$. By Proposition~\ref{phi}, only the zeroth row of the spectral sequence is nonzero, hence it degenerates at the $E_2$-page.
By the K\"unneth theorem, the cohomology of the derived pushforward
 equals to the cohomology of the right side of 
(\ref{otimes}). Since the $E_2$-page consists of a single row, the total cohomology is given by the zeroth direct image, hence coincides with $\Ho{0}$ of the pushforward. This yields the desired isomorphism.
\end{proof}

\subsection{Vanishing cycles of the convolution}
\label{category}

Proposition \ref{conv2} above and its proof show that the functor $H^0(\PP,\jj-)$ 
is a tensor functor  on the category 
of perverse sheaves on $\Gm$ with respect to the tensor
structure given by convolution. 
By Proposition \ref{phi}, this functor is isomorphic to 
vanishing cycles at $1$.
This functor is not faithful and therefore is not a fiber functor. This can be remedied by passing to an appropriate quotient of $\Shs$.

The following theorem is proved in \cite{DLet, DTerPr}. We will not need it and
 cite it for completeness.

\begin{theorem}[\cite{DTerICM,DTerPr, DLet}] Consider the subcategory of $\Shs$ formed by perverse sheaves
with nilpotent monodromy and  take its Serre quotient 
by the subcategory of perverse sheaves smooth on $\Gm$.
This is a Tannakian category with the fiber functor
given by vanishing cycles at $1$ and the tensor product given 
by the multiplicative convolution.
\label{Quotient}
\end{theorem} 

\begin{rem}
The category $\mathrm{Perv}_0$ introduced in \cite[12.6]{Katz+1991} (see also \cite{GMV}) is closely related to the above construction. It provides a splitting of the analogous quotient for the category of perverse sheaves on the affine line with respect to additive convolution.
\end{rem}

\section{Singular fibers}

\subsection{Sheaves without sections supported at $1$}
\label{Without}


We say that a perverse sheaf $\FF \in \Shs$ \emph{has no sections supported at $1$} if
\[
H^0(i_1^! \FF)=0.
\]
Equivalently, $\FF$ has no sections supported at $1$ if the map
\[
\var\colon \Van{1-z}{\FF}\to \Nea{1-z}{\FF}
\]
is injective (see \eqref{canvar}).
Two basic examples are the $*$-extension and the middle extension of a smooth sheaf on $\Gm\setminus\{1\}$.

We say that a perverse sheaf $\FF\in \Shs$  has no 
sections supported at $1$ if $H^0(i_1^! \FF)=0$. Equivalently,   sheaf  has no sections
with support at 1 if the map $\var\colon \Van{1-z}{\FF}\to \Nea{1-z}{\FF}$
induced by (\ref{var}) at point $1$ is injective. Two main examples are the $*$-extension and the middle extension of a smooth sheaf on $\Gm\setminus\{1\}$. 
An important feature of such sheaves is that \eqref{complex} defines a two-term complex representing the vanishing cycles.

\begin{prop}
If one of the perverse sheaves $\EE, \FF \in \Shs$ has no sections supported at $1$, then the convolution $\con{\EE}{\FF}$ has no sections supported at $1$.
\label{support}
\end{prop}

\begin{proof}
We compute $i_1^!(\con{\EE}{\FF})$ using base change for the map $p'_3$.
Let $f\colon p^{\prime -1}_3(1)\to \Ms{5}'$
be the embedding of the fiber. We need to calculate
$$\Ho{0}(p^{\prime -1}_3(1), f^!({p'_5}^*(\overline{\jj\EE}) \otimes {p'_4}^*(\jj\FF))$$.

The fiber $p^{\prime -1}_3(1)$  is the projective line.
Denote by $s$ the embedding of point $(0, \infty, 1,1,1)$ in this line,
and by $t$ the embedding of the complement. 
Denote ${p'_5}^*(\overline{\jj\EE}) \otimes {p'_4}^*(\jj\FF)$ by $\mathcal{A}$ and consider the standard triangle
\begin{equation*}
\begin{tikzcd}[cramped]
s_*s^!f^!{\mathcal{A}}\arrow[r]
& f^!{\mathcal{A}}\arrow[r] 
& t_*t^*f^!{\mathcal{A}}\arrow[r, "+1"]
& {}
\end{tikzcd}
\end{equation*}
Taking cohomology, we get
\begin{equation}
\begin{tikzcd}[cramped]
\Ho{0}(s^!f^!{\mathcal{A}})\arrow[r]
& \Ho{0}(\PP,\, f^!{\mathcal{A}})\arrow[r] 
& \Ho{0}(\PP,\, t_*t^*f^!{\mathcal{A}})\arrow[r, "+1"]
& {}
\end{tikzcd}
\label{hhh}
\end{equation}
where $\PP$ is $p^{\prime -1}_3(1)$. The first term is equal
to $\Ho{0}(i_1^!\EE)\otimes \Ho{0}(i_1^!\FF)$ by the K\"unneth theorem and 
vanishes by the condition of the proposition. One may see that
$t_*t^*f^!{\mathcal{A}}$ is isomorphic to 
$\mathop{j_{\infty!}j_{0*} j_{1*}j_1^*}(\overline{\EE}\otimes\FF)[-2]$,
where $j_1$ is $\Gm\setminus\{1\}\hookrightarrow\Gm$. 
One checks that 
The $i$-th cohomology of this sheaf vanishes for $i \ne 1$, by applying Proposition~\ref{phi} to $\mathop{j_{1*}j_1^*}(\overline{\EE}\otimes\FF)[-1]$.
It follows that the third term in \eqref{hhh} vanishes, and hence the middle term vanishes as well.
\end{proof}

\subsection{Nearby cycles of convolution at 0 via the Verdier specialization}
\label{gluing}

For a perverse sheaf in $\Shs$ without sections supported at $1$, the vanishing cycles at $1$ can be viewed as a subspace of the nearby cycles at $1$. 
For the convolution of such sheaves, we construct a natural embedding
\[
\Van{{1-z}}{\EE}\otimes\Van{{1-z}}{\FF} \longrightarrow \Nea{{z}}{\con{\EE}{\FF}}
\]
using the geometry of the special fiber at $0$.

Consider the fiber of $p_3'$ (equivalently, $p_3$) over $0$. 
It is a nodal stable curve consisting of two components.
The convolution is the shifted pushforward under $p_3'$ of the sheaf
\[
{p'_5}^*(\overline{\jj\EE}) \otimes {p'_4}^*(\jj\FF).
\]
To compute nearby cycles of the convolution, we first compute
\[
\nea{z}{{p'_5}^*(\overline{\jj\EE}) \otimes {p'_4}^*(\jj\FF)}
\]
on the singular fiber, and then take the pushforward.

Consider the diagram
\begin{equation}
s\!:\, * \to p_3^{-1}( 0) \leftarrow \mathbb{A}^1 \cup \mathbb{A}^1 \, :\! t_{1,2}
\label{sing}
\end{equation}
where $s$ be the embedding of the singular point, corresponding to the stable curve of type $(15)(3)(24)$, and $t_i$ be the embedding
of the $i$th complement.
The first component, corresponding to stables curves
of types $(15)(234)$ and the second one --- to $(135)(24)$. 

Introduce  coordinates on the compactification of the first and the second component as follows (see the left part of the picture at the end of Section \ref{poly}):
\begin{equation}
(0, \infty, 1)= ((15)(4)(23), (15)(3)(24) , (15)(2)(34))
\label{left1}
\end{equation}
and:
\begin{equation}
(0, \infty, 1)= ( (15)(3)(24), (13)(5)(24),  (35)(1)(24))
\label{left2}
\end{equation}

Denote by $\EE_0$ (resp. $\FF_0$) the shifted vector bundle on $\Gm$ whose fiber at $a$ is $\Nea{az}{\EE}[1]$ (resp. $\Nea{az}{\FF}[1]$). 
These are identified with the restrictions to $\Gm$ of the Verdier specializations \cite{Ver83} of $\EE$ and $\FF$ along $0$, after identifying the tangent space to $\PP$ at $0$ with $\mathbb{A}^1 = \PP\setminus\{\infty\}$.

Denote ${p'_5}^*(\overline{\jj\EE}) \otimes {p'_4}^*(\jj\FF)$ by $\mathcal{A}$.

One may see that on the first component in coordinates as above ${p'_5}^*(\overline{\jj\EE})$ is 
isomorphic to the restriction of $\overline{\jj\EE}$ and nearby cycles of 
$ {p'_4}^*(\jj\FF)$ along $z$ is isomorphic to $j_{\infty!}\FF_0$.
Since nearby cycles commute with tensor products, we obtain, 
\begin{equation}
t_{1*}t^*_1\nea{z}{{p'_5}^*(\overline{\jj\EE}) \otimes {p'_4}^*(\jj\FF)}=
\overline{\jj(\EE\otimes \FF_0)}
\label{near1}
\end{equation}
and, analogously,
\begin{equation}
t_{2*}t^*_2\nea{z}{{p'_5}^*(\overline{\jj\EE}) \otimes {p'_4}^*(\jj\FF)}=
\jj(\EE_0\otimes \FF)
\label{near2}
\end{equation}
where we identify $\PP$, where right-hand sides of 
(\ref{near1}) and (\ref{near2}) are defined, with closures of our components.

Consider   the standard triangle associated with  stratification
(\ref{sing}):
\begin{equation*}
\begin{tikzcd}[cramped]
s_*s^!\nea{z}{\mathcal{A}}\arrow[r]
& \nea{z}{\mathcal{A}} \arrow[r] 
& t_*t^*\nea{z}{\mathcal{A}} \arrow[r, "+1"]
& {}
\end{tikzcd}
\end{equation*}
Thus, $\nea{z}{\mathcal{A}}$ may be presented as
\begin{equation}
\mathop{cone}(\,
\begin{tikzcd}[cramped, sep=small]
 t_*t^*\nea{z}{\mathcal{A}} \arrow[r]
& s_*s^!\nea{z}{\mathcal{A}}[1]\end{tikzcd}
\,),
\label{cone}
\end{equation}
where
\begin{equation}
 t_*t^*\nea{z}{\mathcal{A}}= t_{1*}t_1^*\nea{z}{\mathcal{A}}\,\oplus\, t_{2*}t_2^*\nea{z}{\mathcal{A}}
\label{sum}
\end{equation}

To construct a map
$
\Van{{1-z}}{\EE}\otimes\Van{{1-z}}{\FF} \to \Nea{{z}}{\con{\EE}{\FF}},
$
we construct compatible extensions on the two components and glue them via the cone description \eqref{cone}. Concretely, we build classes in
$
\Ho{-1}(t_{1*}t_1^*\nea{z}{\mathcal{A}})$ and 
$\Ho{-1}(t_{2*}t_2^*\nea{z}{\mathcal{A}})$
whose images in $s_*s^!\nea{z}{\mathcal{A}}[1]$ coincide.

Combining (\ref{near1}) and (\ref{near2}) with Proposition \ref{phi},
we get
\begin{equation}
\begin{gathered}
\Ho{-1}(\PP,\, t_{1*}t^*_1\nea{z}{\mathcal{A}})=\Van{1-z}{\EE\otimes \FF_0}=\Van{1-z}{\EE}\otimes\Nea{z}{ \FF}\\
\Ho{-1}(\PP,\, t_{2*}t^*_2\nea{z}{\mathcal{A}}) =\Van{1-z}{\EE_0\otimes \FF}=\Nea{z}{\EE}\otimes\Van{1-z}{ \FF}
\end{gathered}
\label{gath}
\end{equation}
where the second equalities are due to the fact that
$\EE_0$ and $\FF_0$ are smooth at $1$ and the corresponding nearby
cycles are their shifted fibers there.

Given $u \otimes v \in \Van{{1-z}}{\EE}\otimes\Van{{1-z}}{\FF}$ ,
using isomorphisms (\ref{gath}) we get elements
\begin{equation}
\begin{gathered}
v\otimes (\Hol{I} (\var(u)))\in \Van{1-z}{\EE}\otimes\Nea{z}{ \FF}=\Ho{-1}(\PP,\, t_{1*}t^*_1\nea{z}{\mathcal{A}})\\
(\Hol{I} (\var(v)))\otimes u \in\Nea{z}{\EE}\otimes\Van{1-z}{ \FF}=\Ho{-1}(\PP,\, t_{2*}t^*_2\nea{z}{\mathcal{A}})
\end{gathered}
\label{gath2}
\end{equation}
where isomorphisms $\Hol{I}\colon \Nea{1-z}{-}\to \Nea{z}{-}$ are as in Proposition \ref{hol}.
Taking the difference between the element given by the first and second lines of
(\ref{gath2}) and taking into account (\ref{sum}) we get an element of
$\Ho{-1}(t_*t^*\nea{z}{\mathcal{A}})$.

To fulfill the construction, we need to check that projections on $s_*s^!\nea{z}{\mathcal{A}}[1]$ in (\ref{cone})
of extensions  given by the first and second lines of
(\ref{gath2}) coincide. To do it, one need to notice,
that these projections are given by the second arrow of the triangle (\ref{zero})
and apply Proposition \ref{frame}.

%
%
%
%

\begin{prop}
The map from $\Van{{1-z}}{\EE}\otimes\Van{{1-z}}{\FF}$
to $\Nea{{z}}{\con{\EE}{\FF}}$ defined above is the composition of maps
\begin{equation}
\begin{tikzcd}[scale=05, cramped, sep=small]
\Van{}{\EE} \otimes \Van{}{\FF}\arrow[r, "\ph{I}\otimes\ph{I}\!\!", "\sim"'] 
& \Ho{0}(\PP,\, \jj \FF)\otimes\Ho{0}(\PP,\, \jj \EE) \arrow[r,"\mbox{\tiny Prop.\ref{conv2}}",  "\sim"'] 
& \Ho{0}(\PP,\, \jj( \con{\EE}{\FF})) \arrow[d]\\
\Nea{{z}}{\con{\EE}{\FF}}
&\Van{{z}}{j_{1*}j_1^*(\con{\EE}{\FF})} \arrow[l, "\sim", "var_0"']
&\hspace{-.7cm}\Ho{0}(\PP,\, \mathop{j_{1*}j_1^*j_{\infty!}j_{0*} }(\con{\EE}{\FF}))
 \arrow[l, start anchor={[shift={(-17pt,0pt)}]}, 
      end anchor={[shift={(0pt,0pt)}]} , "\ph{I}^{-1}"', "\sim"]
\end{tikzcd}
\label{map1}
\end{equation}
where all arrows, except the vertical one, are isomorphisms, and the vertical 
arrow is the canonical map.
\label{Map1}
\end{prop}

\begin{proof}
The statement follows by tracing through the constructions and comparing with Proposition~\ref{phi}.
\end{proof}



\subsection{Nearby and vanishing cycles of convolution at 1  via the Verdier specialization}
\label{blowup}

Analyzing the fibers of the maps $p_3$ and $p_3'$ over $1$, one can identify the vanishing cycles of the convolution with sections of products of Verdier specializations of the corresponding sheaves.
In this subsection, we compute the vanishing cycles at $1$ and their image in nearby cycles for the convolution of perverse sheaves from $\Shs$ without sections supported at $1$.

The fiber $p_3^{-1}(1)$ is a nodal stable curve consisting of two components. Consider the diagram
\begin{equation}
s\!:\, * \to p_3^{-1}(1) \leftarrow \mathbb{A}^1 \cup \mathbb{A}^1 \, :\! t \cup t'.
\end{equation}
Here:
$s$ is the embedding of the singular point corresponding to the stable curve of type $(12)(3)(45)$,
$t$ is the embedding of the smooth locus of the component $(123)(45)$, and
$t'$ is the embedding of the smooth locus of the component $(12)(345)$.

Introduce the coordinates on the compactification of the first and the second component as follows (see the right part of the picture at the end Section \ref{poly}):
\begin{equation}
(0, \infty, 1)= ((13)(2)(45), (23)(1)(45) , (12)(3)(45))
\label{right1}
\end{equation}
and:
\begin{equation}
(0, \infty, 1)= ( (12)(4)(35), (12)(3)(45),  (12)(5)(34))
\label{right2}
\end{equation}

The fiber $p^{\prime-1}_3(1)$ is the projective line. 
The natural projection $p_3^{-1}(1)\to p_3^{\prime-1}(1)$ contracts the component $(12)(345)$ and is an isomorphism on the other component. We keep the notation $s$ and $t$ for the embeddings of the point $(12)(3)(45)$ and its complement in $p_3^{\prime-1}(1)$:

\begin{equation}
s\!:\, * \to p_3^{\prime-1}( 1) \leftarrow \mathbb{A}^1  \, :\! t
\label{sing2}
\end{equation}

Denote $p_5^*(\overline{\jj\EE}) \otimes p_4^*(\jj\FF)$ also by $\mathcal{A}$.

As above, for sheaves in $\Shs$ without sections supported at $1$, we view vanishing cycles at $1$ as a subspace of nearby cycles. Thus, by the Milnor triangle \eqref{var}, to construct an element of $\Van{1-z}{\con{\EE}{\FF}}$, it suffices to construct an element of
 $\Nea{1-z}{\con{\EE}{\FF}}=\Ho{-1}(p_3^{\prime-1}( 1), \nea{1-z}{\mathcal{A}})$,
which vanishes under the connecting morphism of the cohomology
of the Milnor triangle to 
$\Ho{1}(i_1^! (\con{\EE}{\FF}))$.
We  claim that if $\EE$ or $\FF$ has no sections supported at $1$,
for this purpose one can take $v \in \Ho{-1}(p_3^{\prime-1}( 1), \nea{1-z}{\mathcal{A}})$ such that
\begin{equation}
\Ho{-1}(p_3^{\prime-1}( 1),\, t^*\nea{1-z}{\mathcal{A}}) \ni t^*v=0
\label{vanish}
\end{equation}
Indeed, let
$f\colon  p_3^{\prime-1}( 1) \to \Ms{5}'$ be the embedding  and
 consider the triangle associated with stratification (\ref{sing2})
for $f^!\mathcal{A}$:
\begin{equation*}
\begin{tikzcd}[cramped]
s_*s^!f^!{\mathcal{A}}\arrow[r]
& f^!{\mathcal{A}}\arrow[r] 
& t_*t^*f^!{\mathcal{A}}\arrow[r, "+1"]
& {}
\end{tikzcd}
\end{equation*}
Taking cohomology, we obtain:
\begin{equation*}
\begin{tikzcd}[cramped ,sep=small]
\Ho{0}(s^!f^!{\mathcal{A}}) \arrow[r]\arrow[d, equal]
& \Ho{0}(p_3^{\prime-1}( 1), \,f^!{\mathcal{A}}) \arrow[r] 
& \Ho{0}(\mathbb{A}^1,\,t^*f^!{\mathcal{A}}) \arrow[d, equal]\\
\Ho{0}(i_1^!\EE)\otimes \Ho{0}(i_1^!\FF)
& {}
& \Ho{-1}(\mathbb{A}^1,t^*\nea{1-z}{\mathcal{A}})
\end{tikzcd}
\label{tri}
\end{equation*}
Here, the left vertical isomorphism is the K\"unneth theorem
and this term vanishes, because $\EE$ or $\FF$ has no sections supported at $1$. The right isomorphism follows from the fact that
${\mathcal{A}}$ is smooth in the neighborhood
of the image of $t$, hence  $\van{1-z}{\mathcal{A}}$ vanishes there.
Consider the image of $v$ under the connecting isomorphism in the middle term of
(\ref{tri}). Its image under the right arrow vanishes,
because it is equal to $t^*v$, which vanishes by (\ref{vanish}).
The left term is zero, which implies that the image is zero itself.

In order to build an element satisfying condition (\ref{vanish}),
express nearby cycles of $\con{\EE}{\FF}$ in terms of the special
fiber of the projection $p_3$ rather than $p_3'$.
Denote by $\overline{t}\colon \PP \to p_3^{-1}$ the closed embedding of the closure
of the image of $t$. Consider the stratification
\begin{equation}
\overline{t}\!:\, \PP \to p_3^{-1}( 1) \leftarrow \mathbb{A}^1 \, :\!  t'
\end{equation}
and the corresponding triangle for nearby cycles on $p_3^{-1}(1)$
\begin{equation}
\begin{tikzcd}[cramped]
t'_!t^{\prime *}\nea{1-z}{\mathcal{A}}\arrow[r]
& \nea{1-z}{\mathcal{A}} \arrow[r] 
& \overline{t}_*\overline{t}^*\nea{z}{\mathcal{A}} \arrow[r, "+1"]
& {}
\label{tt}
\end{tikzcd}
\end{equation}
For an element of $\Ho{-1}(\PP,\, t'_!t^{\prime *}\nea{1-z}{\mathcal{A}}) $, where $\PP$ denotes the closure of image of
$t'$, consider its image under the first arrow in (\ref{tt}).
Since cohomology of nearby cycles on $p_3^{-1}( 1)$ and $p_3^{\prime-1}( 1)$
coincide and give $\Nea{{1-z}}{\con{\EE}{\FF}}$, the image gives an element of 
$ \Ho{-1}(p_3^{\prime-1}( 1),\, \nea{1-z}{\mathcal{A}})$. We claim that this 
element satisfies condition (\ref{vanish}). Indeed, one may see that
$ t^*\nea{1-z}{\mathcal{A}}$ on both $p_3^{-1}( 1)$ and $p_3^{\prime-1}( 1)$
are canonically isomorphic. But $t^*$ on $p_3^{-1}( 1)$ factors through $\overline{t}^*$,
which annihilates the constructed element due to the exactness of (\ref{tt}).
Thus, an element from  $\Ho{-1}(\PP,\, t'_!t^{\prime *}\nea{1-z}{\mathcal{A}}) $
produces an element in nearby cycles of convolution, which may be lifted
to vanishing cycles.

\begin{prop}
For $\EE$ and $\FF$  from $\Shs$ without section supported at $1$, 
the defined above to  vanishing cycles of $\con{\EE}{\FF}$
lifted from the map to nearby cycles
\begin{equation}
\begin{tikzcd}[cramped]
{}
& \Ho{-1}(\PP,\, t'_!t^{\prime *}\nea{1-z}{p_5^*(\overline{\jj\EE}) \otimes p_4^*(\jj\FF)})  \arrow[d]\arrow[ld, dashed]\\
\Van{{1-z}}{\con{\EE}{\FF}} \arrow[r, "var"]
&\Nea{{1-z}}{\con{\EE}{\FF}}
\end{tikzcd}
\end{equation}
is an embedding.
\label{Ver1}
\end{prop}

\begin{proof}
The statement follows directly from definitions and is left to the reader.
\end{proof}

Now calculate the domain of the map from the  proposition above.

For $\EE, \FF\in\Shs$ denote by
$\EE_1$ and $\FF_1$ the shifted vector bundles on
$\mathbb{A}^1$ which are Verdier specializations (\cite{Ver83})
of $\EE$ and $\FF$ along point $1$ after identification of the tangent space to $\PP$
at $1$ with $\mathbb{A}^1=\PP\setminus\infty$ in the natural way.

One may see that  in coordinates as above
$t^{\prime *}\nea{1-z}{{p'_5}^*(\overline{\jj\EE})}$ is isomorphic to $r^*\EE_1$ and 
$t^{\prime *}\nea{1-z}{ {p'_4}^*(\jj\FF))}$ is isomorphic to $\FF_1$,
where $r$ is defined in (\ref{reflexion}).
Thus, 
\begin{equation}
t^{\prime *}\nea{1-z}{\mathcal{A}}=r^*\EE_1\otimes\FF_1
\label{refl}
\end{equation}

\begin{prop}
For $\EE$ and $\FF$  from $\Shs$ without section supported at $1$, 
\begin{equation}
\Ho{0}(\PP,\, j^\infty_!(r^*\EE_1\otimes\FF_1)[-1])= \Van{{1-z}}{\EE}\otimes\Van{{1-z}}{\FF}
\label{Ver2f}
\end{equation}
\label{Ver2}
\end{prop}

\begin{proof}
We apply Proposition \ref{basic} to calculate the left-hand side of (\ref{Ver2f}).
The standard triangle for the local cohomology on $I$
associated with the stratification
$(\{0, 1\}, I)$
presents it as a complex in the derived category, that is a cone of
\begin{equation}
\begin{tikzcd}[cramped, sep=small]
\Nea{}{\EE_1}\otimes\Nea{}{\FF_1}\arrow[r]
& (\Ho{\bullet}(i_0^!\EE_1)[2]\otimes\Nea{}{\FF_1})\oplus (\Nea{}{\EE_1}\otimes\Ho{\bullet}(i_1^!\FF_1)[2])
\end{tikzcd}
\label{tens}
\end{equation}
Here, we use the identification of the cohomology of
the Milnor triangle (\ref{var}) with the complex given by stratification mentioned in the proof of Proposition \ref{phi}.
Analogously, applying  the canonical isomorphisms 
between nearby and vanishing cycles of perverse sheaves
and their Verdier specializations, one  may present 
the right-hand side of (\ref{Ver2f}) as a product of complexes
\begin{equation*}
\begin{tikzcd}[cramped, sep=small]
\Nea{}{\EE_1}\arrow[r]
& \Ho{\bullet}(i_0^!\EE_1)[2] 
\end{tikzcd} 
\quad \mbox{and} \quad
\begin{tikzcd}[cramped, sep=small]
\Nea{}{\FF_1}\arrow[r]
& \Ho{\bullet}(i_1^!\FF_1)[2]
\end{tikzcd}
\end{equation*}
One may see that there is a natural map from this tensor product
to (\ref{tens}), and the cone of this map is $\Ho{\bullet}(i_0^!\EE)[2]\otimes \Ho{\bullet}(i_1^!\FF)[2]$.
It follows that this map induces an isomorphism on the 0th cohomology,
because
$\EE$ and $\FF$ have no section supported at $1$.
\end{proof}

Combining Propositions \ref{Ver1} and \ref{Ver2} and (\ref{refl}),
we get for  perverse sheaves $\EE, \FF \in \Shs$ without section supported  at $1$ an isomorphism
\begin{equation}
\Van{{1-z}}{\EE}\otimes\Van{{1-z}}{\FF}=\Van{{1-z}}{\con{\EE}{\FF}} 
\label{map2}
\end{equation}


\section{Vanishing cycles and polygons}
\label{poly}

\subsection{Graphical representation}
\label{GrRep}

The isomorphism $\ph{I}$ between vanishing cycles $\Van{1-z}{\FF}$ and cohomology $\Ho{0}(\PP, \jj \FF)$, given by Proposition~\ref{phi}, admits a graphical interpretation in the following sense.
To a vanishing cycle one associates a cocycle in $H^0(I, i_I^!\jj \FF)$, and then pushes it forward to $\PP$.
 Thus, the interval $I$, together with a cocycle representing local cohomology supported on it, may be viewed as a graphical representation of a vanishing cycle.
We now construct a graphical representation of vanishing cycles for the convolution. More precisely, we describe their image in nearby cycles.

For two perverse sheaves $\EE, \FF\in \Shs$, consider the external product of the corresponding cocycles in $H^0_I(\PP,-)$. This yields a cocycle supported on a square in $\Gm\times\Gm$. Under the multiplication map, this square projects onto the interval $I$.
The fiber of this projection over an interior point of $I$ is an interval $[\lambda_1,\lambda_2]$ with $0<\lambda_1<\lambda_2<1$. The corresponding cocycle, viewed as a class in local cohomology supported on this interval (identified with $I$), is given by a section of $\EE\boxtimes\FF$ on this interval, namely the product of the restrictions of the two factors. These sections vanish at the endpoints.

The picture at the end of the text may be helpful here: one should imagine that the right vertical side of the pentagon-like gray figure is contracted, transforming the pentagon into a quadrilateral.

In this way, we obtain a section of the local system $\R^1 m_!(\EE[-1]\boxtimes\FF[-1])$ over the interior of $I$, which is identified there with $\con{\EE}{\FF}$. After pushforward, this defines a class in $\Ho{0}(\PP, j_{1*}j_1^*(\con{\EE}{\FF}))$. Composing with $\ph{I}$, we obtain a class in
$\Van{{1-z}}{j_{1*}j_1^*(\con{\EE}{\FF})}=\Nea{{1-z}}{\con{\EE}{\FF}}$.

\begin{prop}
For $\EE$ and $\FF$  from $\Shs$ without section supported at $1$, 
 the  map defined above
\begin{equation}
\begin{tikzcd}[cramped, sep=small]
\Ho{0}(\PP,\, \jj\EE)\otimes \Ho{0}(\PP,\,\jj\FF)\arrow[r]
& \Nea{{1-z}}{\con{\EE}{\FF}}
\end{tikzcd}
\end{equation}
factors through
\begin{equation}
\begin{tikzcd}[cramped, sep=small]
\var\colon
\Van{{1-z}}{\con{\EE}{\FF}}\arrow[r]
& \Nea{{1-z}}{\con{\EE}{\FF}}
\end{tikzcd}
\end{equation}
and is an isomorphism on the image, which is
$\Van{{1-z}}{\con{\EE}{\FF}}$.
\label{square}
\end{prop}
\begin{proof}
By Proposition \ref{conv2}, the external product gives an isomorphism 
\begin{equation*}
\begin{tikzcd}[cramped, sep=small]
\Ho{0}(\PP,\, \jj\EE)\otimes H^0(\PP,\,\jj\FF)\arrow[r,"\sim"]
&\Ho{0}(\PP,\, \jj(\con{\EE}{\FF}))
\end{tikzcd}
\end{equation*}
By Proposition \ref{phi}, map $\ph{I}$
establishes an isomorphism between  $\Ho{0}(\PP,\,\allowbreak \jj(\con{\EE}{\FF}))$  and 
$\Van{{1-z}}{\con{\EE}{\FF}}$.
Compose it with the composition
\begin{equation}
\begin{tikzcd}[cramped, sep=small]
\Van{{1-z}}{\con{\EE}{\FF}}\arrow[r]
&\Van{{1-z}}{j_{1*}j_1^*(\con{\EE}{\FF})}\arrow[r, equal]
&\Nea{{1-z}}{\con{\EE}{\FF}}
\end{tikzcd}
\label{comp}
\end{equation}
where the arrow is induced by the canonical map
\begin{equation}
\begin{tikzcd}[cramped, sep=small]
\con{\EE}{\FF}\arrow[r]
&j_{1*}j_1^*(\con{\EE}{\FF})
\end{tikzcd}
\label{adj}
\end{equation}
By the very definition of $\ph{I}$, this is 
the map constructed before the proposition.
By Proposition \ref{nea2}, the composite map (\ref{comp}) is equal to $\var$.
By Proposition \ref{support}, 
$\con{\EE}{\FF}$ has no section supported at $1$, thus $\var$ is injective.
\end{proof}
%
%
%
%
%
%
%
%
%

\subsection{Compatibility}

Combining the constructions of the previous section with the graphical interpretation above, we obtain the following result.

\begin{theorem}
For $\EE, \FF \in \Shs$ 
without sections supported at $1$
the following diagram commutes
\begin{equation}
\begin{tikzcd}[cramped, column sep=1ex]
\Ho{0}(\PP, \, \jj j_{1*}j_1^* (\con{\EE}{\FF} ))\arrow[d, "\ph{I}^{-1}\circ r^*"]
&\Ho{0}(\PP, \, \jj (\con{\EE}{\FF} ))\arrow[d, equal, "(\ref{otimes})" ]
\arrow[ddr, bend left, "\ph{I}^{-1}", "\sim"'{sloped}, thin]\arrow[l, "(\ref{adj})"']
&{}\\
\Van{0}{j_{1*}j_1^*(\con{\EE}{\FF})}\arrow[d, "\var"]
&\Ho{0}(\PP,\, \jj\EE)\otimes \Ho{0}(\PP,\,\jj\FF)\arrow[d, equal, "\ph{I}\otimes\ph{I}"] 
&{}\\
\Nea{0}{\con{\EE}{\FF} }
&\Van{{1}}{\EE}\otimes\Van{{1}}{\FF}\arrow[l, "(\ref{map1})"'] \arrow[r, "(\ref{map2})", "\sim"']
&\Van{1}{\con{\EE}{\FF}}
\end{tikzcd}
\label{theorem}
\end{equation}
where  $\Van{{1}}{-}$, $\Van{{0}}{-}$ and $\Nea{{0}}{-}$ denote
vanishing cycles at $1$ along $1-z$ and vanishing and nearby cycles at $0$
along $z$ respectively.
%
\label{Theorem}
\end{theorem}

\begin{proof}
We begin with the right triangle of the diagram. The right isomorphism $\ph{I}$ may be factored through 
\begin{equation*}
\begin{tikzcd}[cramped, sep=small]
\Van{{1-z}}{\con{\EE}{\FF}}\arrow[r]
&\Van{{1-z}}{j_{1*}j_1^*(\con{\EE}{\FF})}\arrow[r, equal]
&\Nea{{1-z}}{\con{\EE}{\FF}}
\end{tikzcd}
\end{equation*}
as in Proposition \ref{square}.
In that proposition, we present elements of
as cohomology with support on $I$, that is as sections of 
$\con{\EE}{\FF}$ over $I$. 
By the very definition of $\ph{I}$ in Proposition \ref{phi}
via \cite{Galligo1985}, the
image of such a class in $\Nea{1}{\con{\EE}{\FF}}$
 is the limit of this section.
If we use the blown-up fiber over $1$ as in Subsection~\ref{blowup}, then in the figure below this section is represented by the right vertical edge of the gray pentagon.
The limit of it is the bold line connecting $4$ and $5$ drawn on the figure presented,  the blown-up fiber over $1$. Thus, the class of nearby cycles
where $\ph{I}$ landed is exactly the class, constructed in 
Proposition \ref{Ver2} by means of the shrinking arguments from Proposition \ref{basic}. 

The left part of the diagram may be treated likewise.
This fact is essentially the content of Proposition \ref{Map1}.
\end{proof}

The theorem gives a map 
from vanishing cycles of the convolution at $1$ to its nearby cycles at $0$.
The following proposition interprets
this map in terms of  holonomy along $I$.

\begin{prop}
In notations of Theorem \ref{Theorem} the following diagram commutes
\begin{equation}
\begin{tikzcd}[cramped]
\Nea{0}{\con{\EE}{\FF} }\arrow[d, equal]
&\Van{{1}}{\EE}\otimes\Van{{1}}{\FF}\arrow[l, "(\ref{map1})"'] \arrow[r, "(\ref{map2})", "\sim"']
&\Van{1}{\con{\EE}{\FF}}\arrow[d, "\var"]\\
\Nea{0}{\con{\EE}{\FF} }\arrow[rr, equal, "\Hol{I}"]
&{}
&\Nea{1}{\con{\EE}{\FF} }
\end{tikzcd}
\end{equation}
where the top row is the bottom row of (\ref{theorem}) and the isomorphism downstairs
is the holonomy along $I$.
\label{HolI}
\end{prop}
\begin{proof}
By the very construction preceding Proposition \ref{square}, images of maps from $\Van{{1}}{\EE}\otimes\Van{{1}}{\FF}$
in nearby cycles coincide with those obtained from
$\Ho{0}(\PP,\, j_{1*}j_1^*(\con{\EE}{\FF}))$ via the isomorphisms $\ph{I}$.
Then, the statement  is a corollary of Proposition \ref{hol}.
\end{proof}

The following picture illustrates the proposition above.
The gray pentagon-like figure is the graphical representation
of the cohomology of the convolution. It may also be viewed as the trace of a cohomology class in the fiber under transport along the interval $I$.

\begin{figure}[hb]
\includegraphics{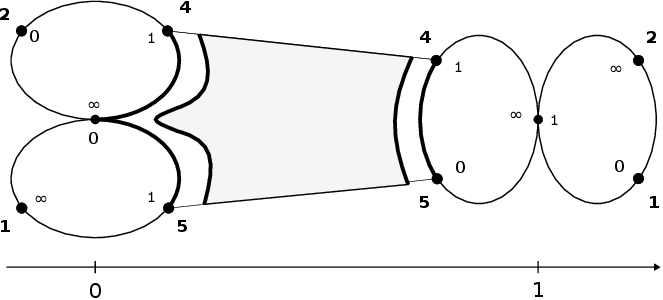}
\centering
\end{figure}


\section{Transport algebra}
\label{Trans}

\subsection{Fundamental groups}

For a topological space $X$ and points $a,b \in X$, we denote by $\fg{X}{a}$
the fundamental group of $X$ with basepoint $a$, and by
$\pa{X}{a}{b}$ the set of homotopy classes of paths from $a$ to $b$.
The set $\pa{X}{a}{b}$ carries the structure of a right torsor over
$\fg{X}{a}$ and a left torsor over $\fg{X}{b}$.

We denote by $\fgu{X}{a}$ the group of $\kf$-points of the pro-unipotent completion of $\fg{X}{a}$.
Equivalently, it is the group of group-like elements in the completed group algebra
$\kf\bra{\fg{X}{a}}$.

In the case of a free group with generators $X_1,\dots,X_n$, its pro-unipotent completion is described by the completed algebra $\kf\bra{X_1,\dots,X_n}$. The completion is taken with respect to the two-sided ideal generated by $X_i - 1$. We refer to this description as \emph{augmentation coordinates}.

On the other hand, using the Lie algebra of $\fgu{X}{a}$, the same object can be described using the completed universal enveloping algebra. In the free case, this identifies with $\kf\bra{\log X_1,\dots,\log X_n}$, completed with respect to the ideal generated by $\log X_i$. We refer to these as \emph{logarithmic coordinates}.

The $\fgu{X}{a}$-torsor $\pau{X}{a}{b}$ is obtained from $\pa{X}{a}{b}$ by a change of structure group along the natural map
$\fg{X}{a} \to \fgu{X}{a}$. Applying the same construction to the left action yields the same torsor, so $\pau{X}{a}{b}$ is also a left torsor over $\fgu{X}{b}$. Elements of $\pau{X}{a}{b}$ are called pro-unipotent paths.
We denote by $\kf\langle \pau{X}{a}{b} \rangle$ the completed $\kf$-vector space spanned by the pro-unipotent paths.

Let $V$ be a (pro-)unipotent local system, i.e.\ a local system equipped with a decreasing filtration such that the associated graded local system is constant. For a pro-unipotent path $g \in \pau{X}{a}{b}$, the holonomy
$\Hol{g}\colon V_a \to V_b$ is defined; it is an isomorphism between the fibers at $a$ and $b$.
In particular, for loops this yields a representation of $\fgu{X}{a}$.
If $a$ and $b$ are points at infinity, the holonomy induces an isomorphism between nearby cycles. For more details, see e.g.\ \cite[\S 9]{DGal}.

The group ring of the pro-unipotent completion of the fundamental group of $\Gms$ is isomorphic to the completed free algebra
$\kf\bra{X_0, X_1}$.
The generator $X_1$ corresponds to the counterclockwise loop around $1$, based at the tangential base point $1 - z$, in $\fgu{\Gms}{1-z}$.
The generator $X_0$ is the conjugate, by $I$, of the counterclockwise loop around $0$, based at the tangential base point $z$.
The corresponding logarithmic coordinates are denoted by $e_i = \log X_i$.

\begin{rem}
The dichotomy between augmented and logarithmic coordinates
corresponds to the Betti and de~Rham realizations in \cite{DLet, DTerICM, DTerPr}.
While this distinction is conceptually important in Hodge theory, in our setting the choice of coordinates is primarily a matter of convenience.
\end{rem}

\subsection{Can, var and Var}

Given a perverse (or constructible) sheaf $\FF$ and a function $f$, the nearby and vanishing cycle functors are related by natural morphisms
$$
\begin{tikzcd}[cramped]
\van{f}{\FF^\bullet}  \arrow[r,"var"]
& \nea{f}{\FF^\bullet}
\end{tikzcd}
\qquad\mbox{and}\qquad
\begin{tikzcd}[cramped]
\nea{f}{\FF^\bullet}  \arrow[r,"can"]
& \van{f}{\FF^\bullet}
\end{tikzcd}
$$
which satisfy the relations
\begin{equation}
\can\circ \var = T-1 \qquad\mbox{and}\qquad \var\circ \can = T-1,
\label{canvar}
\end{equation}
where $T$ denotes the monodromy operator.

If the operator $T$ is unipotent (for instance, if $\FF$ is smooth and unipotent outside the special fiber), we define, following \cite[Sec.~3.4.10]{Saito1988}, a morphism
$$
\begin{tikzcd}[cramped]
\van{f}{\FF^\bullet}  \arrow[r,"Var"]
& \nea{f}{\FF^\bullet}
\end{tikzcd}
$$
by
$$
\Var = \var\circ \frac{\log T}{T-1}
\qquad\mbox{or equivalently}\qquad
\Var = \frac{\log T}{T-1}\circ \var.
$$
By definition, this morphism satisfies the relations
\begin{equation}
\can\circ \Var = \log T \qquad\mbox{and}\qquad \Var\circ \can = \log T.
\label{Var}
\end{equation}
Thus, one may think of $\Var$ as a logarithmic (or de~Rham) version of $var$.

In the setting of this paper, let $\FF \in \Shs$ be a sheaf that is unipotent outside $1$.
Recall (Subsection~\ref{Without}) that the condition that $\FF$ has no sections supported at $1$ means that the map
$
\var\colon \Van{1-z}{\FF}\to \Nea{1-z}{\FF}
$
is injective. Since $\Var$ differs from $\var$ by composition with an automorphism, it is also injective.

The following statement allows one to describe $\can$ in this situation.

\begin{prop}
In the above setting, the compositions of
\[
\can\colon \Nea{1-z}{\FF}\to \Van{1-z}{\FF}
\]
with $\var, \Var\colon \Van{1-z}{\FF}\to \Nea{1-z}{\FF}$
are given by the action of $X_1 - 1$ and $e_1$, respectively, on the nearby cycles.
\label{image}
\end{prop}

\begin{proof}
This follows immediately from \eqref{canvar} and \eqref{Var}.
\end{proof}

\subsection{Transport algebra}
\label{Transport}

By \cite[Ch.~1.12]{jacobson}, for an associative algebra $U$
and an element $a\in U$, the $a$-homotope of $U$ is the algebra (without unit) with the same underlying vector space and product
$
x \cdot_a y = x a y.
$
For an invertible element $i\in U$, the $a$-homotope is isomorphic to the $ia$- and $ai$-homotopes; the isomorphisms are given by left and right multiplication by $i$.

Let $X$ be a Riemann surface, $p \in X$ a point, and $v$ a nonzero
tangent vector at $p$, which we use as a tangential base point
for $\fg{X\setminus\{p\}}{v}$. Denote by $T \in \fg{X\setminus\{p\}}{v}$
the counterclockwise loop around $p$.
We define the \emph{transport algebra} of the pair $(X, p)$ with respect to
the tangential base point $v$ to be the $(T-1)$-homotope of the group algebra $\kf[\fg{X\setminus\{p\}}{v}]$.

For any perverse sheaf $\FF$ on $X$ smooth outside $p$,
the transport algebra acts on the right on $\Van{v}{\FF}$. The action of an element corresponding to
$a \in \kf[\fg{X\setminus\{p\}}{v}]$ is given by the composition
\begin{equation}
\begin{tikzcd}[cramped]
\Van{v}{\FF} \arrow[r,"var"]
&\Nea{v}{\FF}\arrow[r,"\Hol{a}"]
&\Nea{v}{\FF} \arrow[r,"can"]
& \Van{v}{\FF}
\end{tikzcd}
\label{transport}
\end{equation}
Here $\Hol{a}$ is given by the right action of the fundamental group  on the nearby cycles. This defines an action of the transport algebra in view of \eqref{canvar}. 

We define the pro-unipotent transport algebra by replacing the fundamental group with its pro-unipotent completion. 
Additionally, one may introduce a logarithmic version of the pro-unipotent transport algebra. It is defined as the $\log T$-homotope, and its action on vanishing cycles is given by \eqref{transport} with $var$ replaced by $Var$.
This logarithmic transport algebra is isomorphic to the initial one, since $var$ and $Var$ differ by an invertible element.

In the setting of this paper, we take $X = \Gm$, $p = 1$,
and the tangential base point $v = 1-z$, so that
$\fg{X\setminus\{p\}}{v} = \fg{\Gms}{1-z}$.
Following \cite{DTerPr, DTerICM}, denote the corresponding pro-unipotent transport algebra with a unit adjoined by $W$.
One may identify the transport algebra with the left ideal of the group ring generated by $X_1-1$ (equivalently, by $e_1$):
\begin{align}
W&=\kf\cdot 1\;\oplus\; \kf\bra{X_0, X_1}(X_1-1) \label{ideal}\\
W&= \kf\cdot 1\;\oplus\; \kf\bra{e_0, e_1}e_1 \label{logideal}
\end{align}

The structure of $W$ in augmentation coordinates is not obvious (see \cite[Sec.~2.2]{Enriquez2021} for details), but in logarithmic coordinates it is simple.

\begin{prop}
As a topological algebra, $W$ is isomorphic to the completed free algebra generated by
$y_i = -e_0^{i-1}e_1$ for $i\ge 1$:
\begin{equation*}
W = \kf\bra{y_1, y_2, \dots}
\end{equation*}
\label{StructureTr}
\end{prop}

\begin{proof}
This follows by a direct computation.
\end{proof}

\subsection{Harmonic coproduct}

Recall that in Subsection~\ref{blowup} we constructed an isomorphism between the vanishing cycles of the convolution
$
\Van{1-z}{\con{\EE}{\FF}}
$
and the tensor product
$
\Van{1-z}{\EE} \otimes \Van{1-z}{\FF}.
$
Thus, for $\EE, \FF \in \Shsu$, the transport algebra $W$ acts on the tensor product of their vanishing cycles.

The following theorem states that this action is described by a coproduct on the transport algebra, called the \emph{harmonic coproduct}. 
Its formula is most naturally expressed in logarithmic coordinates, in terms of the free generators introduced in Proposition~\ref{StructureTr}. The expression in augmentation coordinates is more complicated (see \cite[Sec.~8.2]{Enriquez2021}), but can also be recovered from the proof below.

\begin{theorem}[\cite{DTerPr, DTerICM}]
For perverse sheaves $\EE, \FF \in \Shsu$ without sections supported at $1$, the action of $w \in W$ on $\Van{1-z}{\con{\EE}{\FF}}$, which is identified with $\Van{1-z}{\EE} \otimes \Van{1-z}{\FF}$ via the above isomorphism, is given by coproduct $\cop(w)$, which is determined on generators (see Proposition~\ref{StructureTr}) by
\begin{equation}
\cop(y_n) = \sum_{i=0}^n y_i \otimes y_{n-i}, \qquad y_0 = 1.
\label{hcoproduct}
\end{equation}
\label{Hcoproduct}
\end{theorem}

\begin{proof}
The proof is essentially a geometric version of \cite[Sec.~6.3]{DTerPr} and \cite[Sec.~5.2]{Enriquez2021}.

To prove the statement, we describe the action of the transport algebra on the vanishing cycles of the convolution. This reduces to computing the operator $\var$, the action of holonomy on nearby cycles, and the action of $\can$, which is given by Proposition~\ref{image}.

It suffices to prove the statement for sheaves of the form $j_{1*}\LL[1]$, where $\LL$ is a pro-unipotent local system on $\Gms$. Indeed, there is a canonical map
$
\EE \to j_{1*}j_1^*\EE
$
which induces an embedding of vanishing cycles. Thus, we may assume that
\[
\var\colon \Van{1}{\EE} \xrightarrow{\sim} \Nea{1}{\EE}, \qquad
\var\colon \Van{1}{\FF} \xrightarrow{\sim} \Nea{1}{\FF}
\]
are isomorphisms, and moreover that the corresponding local systems are universal unipotent.

We identify nearby cycles with the fiber over a point in the smooth locus of $\con{\EE}{\FF}$. The fiber over such a point is a copy of $\PP$, denoted $\PP_\lambda$, with four marked points, which we label $1,2,4,5$.
Denote by $(p_5^*(\overline{\EE}) \otimes p_4^*(\FF))_\lambda$ the restriction of the sheaf to this fiber.
It is smooth outside these points, it is extended by $!$
to points $1$ and $2$, and by $*$ to points $4$ and $5$.
Our aim is to compute the action of the pro-unipotent completion of the fundamental group of the base, that is, of $\pi_1^{un}(\M{4})$  on  the cohomology of 
this restriction.

Consider the canonical map
\begin{equation}
(p_5^*(\overline{\EE}) \otimes p_4^*(\FF))_\lambda\to j_{1*}j^*_{1}(p_5^*(\overline{\EE}) \otimes p_4^*(\FF))_\lambda,
\label{embed}
\end{equation}
where $j_1$ is the open embedding of the complement of point $1$.
If the monodromy around point $1$ has no kernel, which is the case for extension of universal objects as above, this map induces an emedding of cohomology, respected
by action of holonomy. Thus, one needs to calculate action of the holonomy of cohomology of the target of this map.

Introduce notations:
\begin{gather*}
E=\Van{1}{\EE}=\Nea{1}{\EE}\qquad\quad
F=\Van{1}{\FF}=\Nea{1}{\FF}\\
H=\Ho{-1}(\PP_\lambda, \, j_{1*}j^*_{1}(p_5^*(\overline{\EE}) \otimes p_4^*(\FF))_\lambda)
\end{gather*}

Let us connect points $2$, $4$, $5$ by non-intersecting, non-self-intersecting 
oriented intervals $I_{24}$, $I_{25}$, $I_{45}$ such that they form a triangle without marked points inside. As in Proposition \ref{Ver2}, 
cohomology of the restriction $H^{-1}(I_{ij}, \,i_I^! (p_5^*(\overline{\EE}) \otimes p_4^*(\FF))_\lambda)$ isomorphic to $E\otimes F$.
One can choose these isomorphisms compatibly so that the corresponding identifications of fibers extend over the interior of the triangle.

Introduce notations for corresponding pushforward morphisms:
\begin{equation*}
\begin{tikzcd}[cramped]
\intr{ij}\colon  E\otimes F  \arrow[r]\arrow[d, equal]& H \arrow[d, equal]\\
H^{-1}(I_{ij}, \, i_I^! \,(p_5^*(\overline{\EE}) \otimes p_4^*(\FF))_\lambda)  \arrow[r,"i_{I_{ij}*}"]
& H^{-1}(\PP_\lambda, \, j_{1*}j^*_{1}(p_5^*(\overline{\EE}) \otimes p_4^*(\FF))_\lambda)
\end{tikzcd}
\end{equation*}
Note that 
\begin{equation*}
\intr{45}=\intr{42}+\intr{25}
\end{equation*}

Applying Proposition \ref{basic} to 
the interval $I=I_{42}\cup I_{25}$ we get an isomorphism 
\begin{equation}
\nabla\circ(\intr{42}\oplus \intr{25})\colon \,{(E\otimes F)}^{\oplus 2}\to H,
\label{base}
\end{equation}
where $\nabla\colon H\oplus H\to H$ is the codiagonal.
Our aim is to describe action of holonomy on $H$ in terms of this and other analogous  isomorphisms.

Denote by $X_{ij}$ the element in $\pi^{un}_1(\M{4})$ presented by the loop, corresponding to 
the round of $i$-th point around $j$-th one counterclockwise.
Denote $e_{ij}=\log X_{ij}$.

Firstly compute action of $e_{45}$ (corresponding to $e_1$ in Proposition \ref{StructureTr}) on $H$.
The element $X_{45}$ can be represented by a loop such that the marked points do not intersect the interval $I_{45}$. Therefore, its action commutes with $\intr{45}$, 
so it is true for $e_{45}$.
That is
\begin{equation}
e_{45}\intr{45}=\intr{45}e_{45}=\intr{45}(e_1\otimes1+ 1\otimes e_1)
\label{pr1}
\end{equation}

As it follows from Proposition  \ref{Ver2}, image of 
$\intr{45}$ is the image of $\var\colon\Van{1}{\con{\EE}{\FF}}\stackrel{}{\to}\Nea{1}{\con{\EE}{\FF}}$ in the nearby cycles.
So, by Proposition \ref{image}, 
\begin{equation*}
e_{45}\intr{42}=\intr{45}x_1 \qquad\qquad e_{45}\intr{25}=\intr{45}x_2
\end{equation*}
One may see, that the projection on the second factor in $E\otimes F$ of $x_1$
and projection on the first factor of $x_2$ are zero. In combination with
\eqref{pr1} it follows
\begin{equation*}
e_{45}\intr{42}=\intr{45}\,e_1\otimes 1 \qquad\qquad e_{45}\intr{25}=\intr{45}\,1\otimes e_1
\end{equation*}
This result may be derived from the explicit geometric definition of 
the map $\can$ as well.
Thus, the block matrix of action of $e_1$ in the basis \eqref{base} is
\begin{equation}
e_1=\begin{pmatrix}
e_1\otimes 1 & 1\otimes e_1 \\
e_1\otimes 1 & 1 \otimes e_1
\end{pmatrix}
\label{mat1}
\end{equation}

To compute action of $e_{15}$ (corresponding to $e_0$ in Proposition \ref{StructureTr}), 
observe that
\begin{equation*}
e_{15}\intr{42}=\intr{42}e_{15} \qquad\qquad e_{14}\intr{25}=\intr{25}e_{14}
\end{equation*}
because corresponding elements of $\pi^{un}_1(\M{4})$ 
can be represented by a loop such that the marked points do not intersect corresponding intervals, as above. 
We get
\begin{equation*}
e_{15}\intr{42}=\intr{42}e_{15} =\intr{42}\, e_0\otimes 1
\end{equation*}
and, substituting $e_{15}=-e_{45}-e_{14}$, we get
\begin{equation*}
e_{15}\intr{25}= -e_{45}\intr{25} -e_{14}\intr{25}=
-\intr{45}\,1\otimes e_1-\intr{25}e_{14}=
-(\intr{42}+\intr{25})1\otimes e_1 +\intr{25}(1\otimes e_1+1\otimes e_0)
\end{equation*}
Thus, the matrix of action of $e_0$ in the basis \eqref{base} is
\begin{equation}
e_0=\begin{pmatrix}
e_0\otimes 1 & -1\otimes e_1 \\
0 & 1\otimes e_0
\end{pmatrix}
\label{mat2}
\end{equation}

Recall that image of $\var$ is given by vector $(1,1)$.
Substituting matrices  \eqref{mat1} and \eqref{mat2} in the definition
of generators of $W$ from Proposition \ref{StructureTr},  and applying
Proposition \ref{image}, we complete the proof.
\end{proof}

\begin{rem}
As shown in \cite{DTerPr}, the quotient symmetric monoidal category from Theorem~\ref{Quotient}, equipped with the fiber functor given by vanishing cycles, is equivalent to the category of $W$-modules. The Tannakian formalism then endows $W$ with a Hopf algebra structure.
To relate this Hopf structure to the theorem above, note that, by the explicit construction of the equivalence (\cite[Prop.~5.1]{DTerPr}), every object in the quotient category is equivalent to a sheaf without sections supported at $1$.

\end{rem}

\section{Fox derivative and semi-holonomy}
\subsection{Fox derivative}

Fox derivatives were introduced in \cite{Fox1953}. 
In a geometric context, they also appear as ``lifting maps'' in 
\cite[Th.~II.5.3, Ex.~II.5.3]{KBrown}.
See also \cite[Sec.~2.1]{Massuyeau2013}.

Let $A$ be a $\kf$-algebra with augmentation 
$\varepsilon\colon A\to \kf$. A (left) Fox derivative
is a map $D\colon A\to A$ such that
\begin{equation}
D(ab)=D(a)\aug{b} + a\,D(b) \qquad \text{and} \qquad D(1)=0.
\label{der}
\end{equation}
Left Fox derivatives form a right $A$-module.

For the group algebra of a free group $\kf[X_1^{\pm 1},\dots, X_n^{\pm 1}]$ 
with the standard augmentation, Fox derivatives are determined by their values on generators. The Fox
partial derivative $\pdv{}{X_i}$ is defined by 
\begin{equation*}
\pdv{X_j}{X_i}=\delta_{ij}, \qquad 
\pdv{X_j^{-1}}{X_i}=-\delta_{ij} X_j,
\end{equation*}
where $\delta_{ij}$ denotes the Kronecker symbol.

An explicit formula for this derivative, given in \cite[(2.5)]{Fox1953}, is as follows.
Writing a group element $u \in \kf[X_1^{\pm 1},\dots, X_n^{\pm 1}]$
in the form
\begin{equation*}
 u=u_0X_j^{p_1}u_1 X_j^{p_2}\cdots u_{k-1}X_j^{p_k}u_k,
\end{equation*}
where the reduced words representing $u_i$ do not involve the generator $X_j$, we obtain
\begin{equation}
\pdv{u}{X_j} = \sum_{i=1}^{k} u_0 X_j^{p_1} \cdots u_{i-1} \frac{X_j^{p_i}-1}{X_j-1}.
\label{fox}
\end{equation}

One checks that the map $f \mapsto f - \aug{f}$ is a Fox derivative,
which acts on generators by $X_i \mapsto X_i - 1$.
This implies the fundamental formula of Fox calculus:
\begin{equation}
f - \aug{f} = \sum_i \pdv{f}{X_i}(X_i - 1).
\label{fund}
\end{equation}

A consequence of the Leibniz rule \eqref{der} is that Fox derivatives lower the augmentation degree by one. It follows that the standard Fox derivatives on $\kf[X_1^{\pm 1},\dots, X_n^{\pm 1}]$
extend naturally to the pro-unipotent completion
$\kf\bra{X_1, \dots, X_n}$.

Denote by $e_i = \log X_i$ the logarithmic coordinates.
The corresponding Fox partial derivative $\pdv{}{e_i}$ is defined by
\begin{equation*}
\pdv{e_j}{e_i} = \delta_{ij}.
\end{equation*}

For the same reason as above, we obtain the logarithmic version of the fundamental formula:
\begin{equation}
f - \aug{f} = \sum_i \pdv{f}{e_i} e_i.
\label{log-fund}
\end{equation}

Comparing \eqref{fund} with \eqref{log-fund}, we obtain in particular that
\begin{equation}
\pdv{f}{X_i}(X_i - 1) = \pdv{f}{e_i} e_i,
\label{fund2}
\end{equation}
since both sides define the projection onto the left ideal generated by $(X_i - 1)$, equivalently by $e_i$.

\subsection{Semi-holonomy}
\label{Semi-holonomy}

Recall that in Subsection~\ref{van1} we defined an isomorphism
\[
\begin{tikzcd}[cramped, sep=small]
\ph{I} \colon \Van{1-z}{\FF} \arrow[r] & \Ho{0}(\PP,\, \jj \FF)
\end{tikzcd}
\]
for a perverse sheaf $\FF \in \Shs$, given by shrinking $\mathbb{A}^1$ on the interval. 
In Subsection~\ref{GrRep}, we interpreted this isomorphism as a geometric realization: one associates to a vanishing cycle a cocycle in $H^0(I,\, i_I^! \jj\FF)$ and then pushes it forward via $i_{I*}i_I^! \jj\FF$ to obtain a cocycle in $\Ho{0}(\PP,\, \jj \FF)$.

In this construction, one may replace the pushforward along the standard embedding of the interval by the pushforward along an arbitrary embedding determined by a path $P\colon I \to \Gms$ from the tangential base point $1-z$ to the tangential base point $z$. The result is invariant under homotopy of the path, and thus depends only on the corresponding element of the fundamental groupoid.

Any such homotopy class of paths can be represented as a composition $gI$, where $g \in \fg{\Gms}{1-z}$ and $I$ is the straight path. Denote by
\begin{equation}
\begin{tikzcd}[cramped, sep=small]
\ph{gI} \colon \Van{1-z}{\FF} \arrow[r] & \Ho{0}(\PP,\, \jj \FF)
\end{tikzcd}
\label{semi-holonomy}
\end{equation}
the morphism given by pushforward along this path, and call it the \emph{semi-holonomy} along $g$.

Although for general $g$ this morphism is not given by shrinking, the following theorem shows that it is an isomorphism in the pro-unipotent setting.

\begin{theorem}
For $g \in \fgu{\Gms}{1-z}$ and $\FF \in \Shsu$, the map \eqref{semi-holonomy} is:
\begin{enumerate}
\item\label{1} well-defined, i.e.\ continuous with respect to the pro-unipotent filtration;
\item\label{2} an isomorphism;
\item\label{3} given by the formula
\begin{equation}
\ph{gI}(-)=\ph{I}(-\tilde{g}),
\label{form}
\end{equation}
where $\tilde{g} \in W$ acts on vanishing cycles as in \eqref{transport}, and, with respect to the identification \eqref{ideal},
\begin{equation*}
\tilde{g}=1+\pdv{g}{X_1}(X_1-1)
\end{equation*}
\end{enumerate}
\label{Semi-holonomy-th}
\end{theorem}

\begin{proof}
({\ref{1}}) This follows from~({\ref{3}}), since the formula is expressed in terms of Fox derivatives, which extend to the pro-unipotent completion.

({\ref{2}}) This also follows from~({\ref{3}}), since the formula shows that $\ph{gI}$ is the sum of the identity and an element of the augmentation ideal.

({\ref{3}}) We first describe the morphisms $\var$ and $\can$ in terms of the geometric representation, following the approach of \cite{Galligo1985, KapSch}. Recall that the nearby cycles of $\FF$ are isomorphic to the vanishing cycles of $j_{1*}j_1^*\FF$. Graphical representation of vanishing cycles of sheaf $j_{1*}j_1^*\FF$
is given by a section of the corresponding local system over the interior of the interval $I$. By Proposition~\ref{nea2}, the morphism $\var$ is given by restriction of such a section to the interior of $I$.

We now describe the morphism $\can$. Nearby cycles are represented by sections over the interior of $I$. To compute the image under $\can$, consider a path consisting of three segments: it starts at the tangential base point $z$, follows $I$ to a point $1-\epsilon$ (for $\epsilon>0$ small), then traverses the circle $|1-z|=\epsilon$ counterclockwise, and finally returns along $I$. Restrict the section corresponding to a nearby cycle to the third segment and extend it along the entire path. The pushforward along this path defines an element of $\Ho{0}(\PP,\,\jj \FF)$, which, by definition, corresponds to $\can$ under the isomorphism $\ph{I}$.

Thus, the composition $\can \circ \var$ sends a vanishing cycle represented by a section over $I$ to a section supported on the loop encircling $1$ as above.

Now write $g \in \fg{\Gms}{1-z}$ as a product of generators, namely the counterclockwise loop around $1$ and the conjugate by $I$ of the counterclockwise loop around $0$. This expresses the difference $\ph{gI} - \ph{I}$ as a sum of contributions corresponding to such elementary loops. One obtains
\begin{equation*}
\ph{X_1^{\varepsilon_1}X_0^{i_1}\cdots X_1^{\varepsilon_k}X_0^{i_k} I}(-)
=
\ph{I}(-)
+
\sum_{i} \ph{I}\bigl(X_1^{\varepsilon_1}X_0^{i_1}\cdots X_0^{i_{i-1}}\,\partial X_1^{\varepsilon_{i}}(-)\bigr),
\end{equation*}
where $\varepsilon_i=\pm 1$, and
\[
\partial X_1^{\varepsilon}=
\begin{cases} 
\can\circ \var & \text{if } \varepsilon=1,\\ 
\can \circ T^{-1}\circ \var & \text{if } \varepsilon=-1.
\end{cases}
\]

Combining this with \eqref{fox}, \eqref{canvar}, and the identification \eqref{ideal}, we obtain the desired formula.
\end{proof}

The following statement explains the geometric meaning of the map $-_Y$ from \cite{DTerPr,DTerICM}.

\begin{prop}
Let $f \in \kf\bra{e_0, e_1}$. Following \cite{DTerPr,DTerICM}, write it in the form $f = 1 + \varphi_0 e_0 + \varphi_1 e_1$ and set
\begin{equation*}
f_Y = 1 + \varphi_1 e_1
\end{equation*}
If $f$ represents $g \in \fgu{\Gms}{1-z}$ in logarithmic coordinates, then $f_Y$ represents $\tilde{g}$ from \eqref{form}.
\label{Y}
\end{prop}

\begin{proof}
The statement follows from \eqref{log-fund}, \eqref{fund2}, and Theorem~\ref{Semi-holonomy-th}.
\end{proof}

The following statement is an analogue of Proposition~\ref{hol}.

\begin{prop}
Let $\LL$ be a pro-unipotent local system on $\Gm \setminus \{1\}$ and let $g \in \fgu{\Gms}{1-z}$. The composition of isomorphisms
\[
\begin{tikzcd}[cramped]
\Nea{1-z}{j_{1*}\LL[1]}
& \Van{1-z}{j_{1*}\LL[1]} \arrow[l, "var"'] \arrow[r, "\ph{gI}"]
& \Ho{0}(\PP,\, j_{\infty!}j_{0*} j_{1*}\LL[1]) \arrow[d,"r"] \\
\Nea{z}{j_{1*}\LL[1]}
& \Van{z}{j_{1*}\LL[1]} \arrow[l, "var"'] \arrow[r, "\ph{gI}"]
& \Ho{0}(\PP,\, j_{\infty!}j_{1*} j_{0*}\LL[1])
\end{tikzcd}
\]
is the holonomy along $gI$.
\label{hol2}
\end{prop}

\begin{proof}
The proof is analogous to that of Proposition~\ref{hol}.
\end{proof}

The above proposition provides another proof of the fundamental relation \eqref{fund} in this setting.

\section{Pentagon equation and double shuffle relations}

\subsection{Pentagon equation}

Denote by $X_{ij}$ the standard generators of $\pi_1(\M{5})$,
corresponding to the loop under which the points i and j exchange counterclockwise while the other points remain fixed.

For an element $f(X_0, X_1)\in \pi_1^{un}(\M{4})$,
we introduce the notation
\begin{equation}
\pi_1^{un}(\M{5})\ni \pen(f)=f(X_{12}, X_{23})f(X_{34}, X_{45})f(X_{51}, X_{12})f(X_{23}, X_{34})f(X_{45}, X_{51}).
\label{assoc}
\end{equation}
Following \cite{Drin}, the relation
\begin{equation}
\pen(f)=1
\label{pent}
\end{equation}
is called the \emph{pentagon equation} for $f$.

We say that $f(X_0, X_1)\in \pi_1^{un}(\M{4})$
is \emph{symmetric} if
\begin{equation}
f(X_0, X_1)=f(X_1, X_0)^{-1}.
\label{sym}
\end{equation}

Following \cite{Fu}, we say that $f(X_0, X_1)\in \pi_1^{un}(\M{4})$
is \emph{commutator group-like}, if it is
group-like and its
coefficient of $X_0$ and $X_1$ are both $0$.

\begin{prop}[{\cite[Lemma 6]{Fu}}]
If commutator group-like element $f(X_0, X_1)\in \pi_1^{un}(\M{4})$ satisfies the pentagon equation
\eqref{pent}, then it is symmetric.
\label{sympent}
\end{prop}

\begin{proof}
Applying the map
$p_i\colon \M{5}\to \M{4}$ forgetting the $i$-th point to \eqref{assoc}, we obtain
\begin{equation}
p_{i*}\pen(f)=f(X_0, X_1)f(X_1, X_0).
\label{pen-sym}
\end{equation}
If $f$ satisfies the pentagon equation, then $\pen(f)=1$, hence
$p_{i*}\pen(f)=1$. Therefore
\[
f(X_0, X_1)f(X_1, X_0)=1,
\]
which is exactly \eqref{sym}.
\end{proof}

\begin{rem}
Symmetry is the simplest case of hexagon relations, see \cite{Drin}.
According to  \cite{Fu}, hexagon relations are implied by the pentagon equation.
\end{rem}

%
%
%

The composition of $f\in \fgu{\M{4}}{1-z}$ with the interval $I\in \pau{\M{4}}{1-z}{z}$ gives a path $fI$ from the tangential base point $1-z$ to $z$ in the pro-unipotent fundamental groupoid of $\M{4}$. 
The pentagon equation is equivalent to the vanishing of the composition of five paths in the pro-unipotent fundamental groupoid of $\M{5}$, obtained as the images of this path under five different tangential embeddings of $\M{4}$ into $\M{5}$ (see \cite{Drin} for details).

The following proposition reformulates this in terms of the holonomy of the local system corresponding to the relative fundamental groupoid.

We use the notation from Subsection~\ref{moduli}. Denote by
\begin{equation*}
F_\lambda=\{(0, \infty, 1, t_1, \lambda t_1)\mid t_1\neq 0, \infty, 1, 1/\lambda\}
\end{equation*}
the fiber of the projection $p_3\colon \M{5}\to \M{4}$ over $(0, \infty, 1, \lambda)\in \M{4}$.

Let $t_2-1$ and $1-t_1$ be tangential base points at
\begin{equation*}
(0, \infty, 1, 1/\lambda, 1)
\qquad\text{and}\qquad
(0, \infty, 1, 1, \lambda),
\end{equation*}
which project to $z-1$ and $1-z$ under $p_4$ and $p_5$, respectively.

The space of pro-unipotent paths in $F_\lambda$ from $t_2-1$ to $1-t_1$
forms a pro-unipotent local system
$
\kf\langle \Pau{F_\lambda}{t_2-1}{t_1-1} \rangle
$
over $\M{4}$.

Let $f\in \fgu{\M{4}}{1-z}$, and let $fI$ be the corresponding pro-unipotent path from $1-z$ to $z$. Consider paths between   marked points $4$ and $5$ in the fibers of the projection
$
p_3\colon \M{5}\to \M{4}
$
over $0$ and $1$. The first path is obtained by composing the images of $fI$ under the identifications \eqref{left1} and \eqref{left2} of the components of the singular fiber with $\M{4}$. The second path is the image of $fI$ under the identification \eqref{right2}.

Denote by $P_0$ and $P_1$ the corresponding nearby cycles of
$
\kf\langle \Pau{F_\lambda}{t_2-1}{t_1-1} \rangle
$
over $1-z$ and $z$, respectively, uniquely determined by the condition that they specialize to these two paths.

\begin{prop}
For a symmetric $f\in \fgu{\M{4}}{1-z}$, the nearby cycles $P_0$ and $P_1$ as above are related by holonomy along $fI$ if and only if $f$ satisfies the pentagon equation \eqref{pent}.
\label{Holonomy1}
\end{prop}

\begin{proof}
Let $p\colon X\to Y$ be a Serre fibration and $s_1,s_2\colon Y\to X$ two sections. Then one has a local system of (pro-unipotent) paths whose fiber over $y\in Y$ is $\pa{p^{-1}(y)}{s_1(y)}{s_2(y)}$. Denote by $i_y\colon p^{-1}(y)\to X$ the inclusion.

For $P\in\pa{p^{-1}(y_1)}{s_1(y_1)}{s_2(y_1)}$ and $f\in\pa{Y}{y_1}{y_2}$, the image in $\pa{X}{s_1(y_2)}{s_2(y_2)}$ of the holonomy of $P$ along $f$ is given by
\begin{equation}
i_{y_2*}\Hol{f} (P)=s_{1*}(fI)^{-1}\circ i_{y_1*} P\circ s_{2*} fI.
\label{formula}
\end{equation}

In our setting, with $X=\M{5}$ and $Y=\M{4}$, the map $i_*$ is injective, so \eqref{formula} determines the holonomy.
Combining this formula with the symmetry condition \eqref{sym}, one obtains the pentagon equation \eqref{pent}.
\end{proof}

\subsection{Homological pentagon equation}

Let $\mathbf{E}$ and $\mathbf{F}$ be pro-unipotent representations
of $\pi_1(\M{4})$, defining local systems on $\M{4}$,
and let $f\in\pi^{un}_1(\M{4})$ be symmetric.

Then by \eqref{pen-sym}, the element $\pen(f)$
acts trivially on $p^*_4\mathbf{E}\otimes p_5^*\mathbf{F}$.
Consider the pro-unipotent loop in $\M{5}$, corresponding to 
$\pen(f)\in\pi_1^{un}(\M{5})$, see \eqref{assoc}. Because $f$ is symmetric,
$p_{3*}\pen(f)=1$ by \eqref{pen-sym}, so 
this loop can be represented by a loop lying in a general fiber
$F_\lambda$ of projection $p_3$.
Since the monodromy of $p^*_4\mathbf{E}\otimes p_5^*\mathbf{F}$ along this loop is trivial, this loop  defines a homology  class with coefficients in the local system 
\begin{equation*}
[\pen(f)]_{p^*_4\mathbf{E}\otimes p_5^*\mathbf{F}}\in H_1(F_\lambda,\, p^*_4\mathbf{E}\otimes p_5^*\mathbf{F})
\end{equation*}
We say that a symmetric element $f$ satisfies the \emph{homological pentagon equation}
if this class vanishes.

This condition has a simple group theoretic meaning.

\begin{prop}
Denote the intersection of kernels of projections $p_i\colon \pi_1^{un}(\M{5})\to \pi_1^{un}(\M{4})$
\begin{equation*}
K = \bigcap_{i=3,4,5}\ker p_{i*} 
\end{equation*}
A symmetric $f \in \pi_1^{un}(\M{4})$
satisfies  the homological pentagon equation if and only if
\begin{equation*}
\pen(f)\in [K, K]
\end{equation*}
\label{Hopent1}
\end{prop}

\begin{proof}
It suffices to prove the statement for the universal pro-unipotent local system
$\Univer{4}$ on $\M{4}$.
The homological pentagon equation asserts that the class
$$
[\pen(f)]_{p_4^*\Univer{4}\otimes p_5^*\Univer{4}}\in H_1(F_\lambda, \, p_4^*\Univer{4}\otimes p_5^*\Univer{4})
$$
vanishes.

Since $\pi_1^{un}(F_\lambda)=\ker p_{3*}$, a standard application of the
Hochschild--Serre spectral sequence identifies
$
H_1\bigl(F_\lambda,\,
p_4^*\Univer{4}\otimes p_5^*\Univer{4}\bigr)
$
with the abelianization of
\begin{equation*}
\ker (p_{4*}\times p_{5*}) \colon
\pi_1^{un}(F_\lambda)\to
\pi_1^{un}(\M{4})\times\pi_1^{un}(\M{4}),
\end{equation*}
where $p_{4*}$ and $p_{5*}$ denote the restrictions of the corresponding maps on
$\pi_1^{un}(F_\lambda)\subset \pi_1^{un}(\M{5})$.
This proves the claim.

Conversely, since the statement holds for the regular representation, it follows for every representation.
\end{proof}

Theorem \ref{Theorem} and  Proposition \ref{HolI} establish certain compatibility 
between semi-holonomy along $I$, i.e. isomorphisms $\ph{I}$ from \eqref{equality},
and multiplicative convolution. The following proposition
states that homological pentagon equation for $f$ is equivalent to the same
compatibility for $\ph{fI}$.

Recall that in Subsections \ref{blowup} and \ref{gluing} we construct 
for $\EE, \FF\in \Shs$
maps from $\Van{1}{\EE}\otimes\Van{1}{\FF}$ to
nearby cycles of $\con{\EE}{\FF}$  at 
$1$ and $0$ respectively. 
This construction produces cohomology classes
on the component of the singular fiber in the first case, and on both components of the singular fiber, followed by gluing, in the second.
This cohomology class was identified with the product of the vanishing cycles
via isomorphism $\ph{I}$. One may reproduce this construction,
replacing isomorphism $\ph{I}$ with $\ph{fI}$ as in Subsection \ref{Semi-holonomy}.
Thus, these classes are represented by sections supported on the paths
$P_1$ and $P_0$ from the previous subsection.

\begin{prop}
For any $\EE, \FF\in \Shs$ without sections supported at $1$ and a
symmetric $f \in \pi_1^{un}(\M{4})$
the maps constructed above from $\Van{1}{\EE}\otimes\Van{1}{\FF}$ to
nearby cycles of $\con{\EE}{\FF}$  at 
$1$ and $0$ respectively commute with holonomy along $fI$,
connecting nearby cycles at $1$ and $0$, if and only if $f$
satisfies the homological pentagon equation.
\label{Hopent2}
\end{prop}

\begin{proof}
The arguments below essentially follow those of the proof of Proposition
\ref{gluing}.

As before, one can consider only perverse sheaves isomorphic to 
$j_{1*}L[1]$ for a local system $L$ on $\Gms$.
We may assume that this local system is the universal pro-unipotent one
as in the proof of the previous proposition. 

The difference between the constructed nearby cycle at $1$ and 
the image of nearby cycle at $0$ under holonomy along $fI$
in the fiber $F_\lambda$ infinitesimally close to $1$
is a cohomology class of the convolution of local systems, extended by $!$ to $\PP$. It is represented by a cocycle supported on a loop.
By the condition, 
the class of this loop with coefficients in the local system is homologous to zero.
By Verdier duality, it is equivalent to the vanishing of the difference.
\end{proof}

\subsection{Double shuffle relations} 

Let $f\in\kf\bra{e_0, e_1}=\pi^{un}_1(\M{4})$ be an element written in logarithmic coordinates. Recall (Proposition \ref{Y}) that if
$f = 1 + \varphi_0 e_0 + \varphi_1 e_1$, then, by definition,
\begin{equation*}
f_Y = 1 + \varphi_1 e_1\in W = \kf\bra{y_1, y_2, \dots},
\end{equation*}
where $W$ is the transport algebra \eqref{logideal}, see Proposition \ref{StructureTr}.

Following \cite{Racinet, IKZ}, an element $f$ satisfies 
\emph{regularised double shuffle relations} 
if the following identity holds in $W\otimes W$:
\begin{equation}
f^{\ab}_Y\cdot\cop(f_Y)=f_Y\otimes f_Y
\label{dsr}
\end{equation}
where we set
\begin{equation}
f_Y^{\ab}=f_Y(y_1\otimes 1, 1\otimes y_1)
\label{ab}
\end{equation}
and $\cop$ is the harmonic coproduct \eqref{hcoproduct}.

\begin{rem}
A more standard formulation of the regularized double shuffle relations is in terms of group-likeness 
of $f_Y^{mod}$ with respect to $\cop{}$,
where $f_Y^{mod}$ differs by a gamma-factor:
\begin{equation*}
W\ni f_Y^{mod}=\Gamma^{-1}(y_1)f_Y
\end{equation*}
Thus, the term $f_Y^{\ab}$ is an analogue of the beta-function
\begin{equation*}
f_Y^{\ab}=\frac{\Gamma(y_1\otimes 1)\Gamma(1\otimes y_1)}{\Gamma(1\otimes y_1+y_1\otimes 1)}
\end{equation*}
It is known that if $f$ satisfies the pentagon equation, then such gamma-factors exist. The geometric meaning of this modified element is not clear to the author.
\end{rem}

\begin{theorem} The following three conditions on a symmetric element
$f\in\pi^{un}_1(\M{4})$ are equivalent:
\begin{enumerate}
\item\label{51} $\pen(f)\in [K, K]$, where $K = \bigcap_{i=3,4,5}\ker p_{i*}$;
\item\label{52} homological pentagon equation;
\item\label{53} regularized double shuffle relations.
\end{enumerate}
\end{theorem}

\begin{proof}
The equivalence of conditions (\ref{51}) and (\ref{52})
is Proposition \ref{Hopent1}. 

To prove the equivalence of (\ref{52}) and (\ref{53}),
we use Proposition \ref{Hopent2}.
Consider the convolution of $\EE$ and $\FF$, where $\EE, \FF\in \Shs$ have no sections supported at $1$.
As before, it suffices to consider perverse sheaves isomorphic to 
$j_{1*}L[1]$ for a local system $L$ on $\Gms$, and we may assume that $L$ is the universal pro-unipotent local system.

Proposition \ref{Hopent2} implies the commutativity of the following diagram, which is analogous to that in Proposition \ref{HolI}.
\begin{equation*}
\begin{tikzcd}[cramped]
\Nea{0}{\con{\EE}{\FF} }\arrow[r, "\ph{fI}"]
&\Ho{0}(\PP,\, \jj(\con{\EE}{\FF})) \arrow[d, "\eqref{conv2}"]\arrow[d, "\sim"', sloped]
&\Nea{1}{\con{\EE}{\FF}}\arrow[l, "\ph{fI}"']\\
\Van{1}{\EE}\otimes\Van{1}{\FF}\arrow[u]\arrow[r, "\ph{fI}\otimes\ph{fI}"]\arrow[bend right=12,equal]{rr}
& \Ho{0}(\PP,\, \jj \EE)\otimes \Ho{0}(\PP,\, \jj \FF)
&\Van{1}{\EE}\otimes\Van{1}{\FF}\arrow[u] \arrow[ul, dashed]\arrow[l, dashed]
\end{tikzcd}
\end{equation*}
Here the upper arrows are given by Proposition \ref{Hopent2}.

Consider the class of the convolution represented by
$\ph{fI}(v)\boxtimes\ph{fI}(v')$, where $v$ and $v'$ are vanishing cycles.
It is straightforward to check that the corresponding nearby cycle at $0$ is the one associated with $v\otimes v'$, constructed as in Proposition \ref{Hopent2}.
This justifies the commutativity of the left square of the diagram.

By Proposition \ref{Y} and Theorem \ref{Hcoproduct}, the sloped dashed arrow is obtained by composing the vertical identifications in diagram \eqref{theorem} with the operator $f^{\ab}_Y\cdot\cop(f_Y)$. Here the first factor accounts for the comparison between $\ph{I}$ and $\ph{fI}$ for the Verdier specialization of 
$p_5^*(\overline{\jj\EE}) \otimes p_4^*(\jj\FF)$
on the blown-up component, while the second factor accounts for the corresponding comparison for the convolution.

Using diagram \eqref{theorem} once more, we find that the left dashed arrow is equal to
\begin{equation}
(\ph{I}\otimes\ph{I})\circ (f^{\ab}_Y\cdot\cop(f_Y))
\label{last1}
\end{equation} 
Similarly, by Proposition \ref{Y}, we obtain
\begin{equation}
\ph{fI}\otimes\ph{fI}=(\ph{I}\otimes\ph{I})\circ (f_Y\otimes f_Y)
\label{last2}
\end{equation} 
Combining \eqref{last1}, \eqref{last2}, and the commutativity of the diagram, we obtain \eqref{dsr} and complete the proof.
\end{proof}	

Combining the theorem above with Proposition~\ref{sympent}, we obtain the following corollary.

\begin{cor*}
For a commutator group-like $f\in \pi_1^{un}(\M{4})$, the pentagon equation implies the regularized double shuffle relations.
\end{cor*}

\section*{Appendix: Homological definition of the self-intersection map}
Let $S$ be a compact oriented surface, possibly with boundary,
equipped with a trivialization of its tangent bundle.
Let $\tau$ be a nowhere-vanishing vector field on $S$
defining this trivialization.

For points $a,b\in S$, possibly lying on the boundary, let
\begin{equation*}
\PA{a}{b}=\kf\langle\pa{S}{a}{b}\rangle
\end{equation*}
denote the $\kf$-vector space spanned by homotopy classes of paths from $a$ to $b$.
Similarly, let
\begin{equation*}
\PAu{a}{b}=\kf\langle\pau{S}{a}{b}\rangle
\end{equation*}
denote the completed $\kf$-vector space spanned by homotopy classes of paths from $a$ to $b$.
The (pro-unipotent) local systems $\PA{a}{-}$ and $\PA{-}{b}$ 
(respectively, $\PAu{a}{-}$ and $\PAu{-}{b}$)
over $S$ are defined in the evident way.
Note that local systems $\PAu{a}{-}$  and $\PAu{-}{a}$  are canonically identified after passing from right to left actions via inversion in the fundamental groupoid. We nevertheless retain  notations $\PAu{a}{-}$  and $\PAu{-}{a}$  to keep track of the natural right and left module structures.

Let $x\colon [0, 1]\to S$ be an immersion of general position satisfying the following condition
\begin{equation}
\begin{gathered}
x(0)=a,\qquad x(1)=b,\qquad
x'(0)=\tau_a,\qquad x'(1)=\tau_b,
\\
\text{and the rotation number of }x'
\text{ with respect to }\tau\text{ is zero.}
\end{gathered}
\label{condition}
\end{equation}
Every homotopy class of paths from $a$ to $b$
may be represented by such an immersion, see \cite[Remark 5.1]{HainTur}.

Let $S^{[2]}$ denote the manifold with boundary obtained
from $S^2$ by replacing the diagonal $\Delta\subset S^2$
with the unit sphere bundle of its normal bundle.
Equivalently, $S^{[2]}$ is the real-oriented blow-up
of $S^2$ along $\Delta$.
The boundary component corresponding to the diagonal
is naturally identified with the unit tangent bundle $S(TS)$.
Let $p_i\colon S^{[2]}\to S$ denote the projection onto the $i$-th factor.

Given an immersion satisfying  condition \eqref{condition}, we define
three 1-chains $c_i$ in $S^{[2]}$ with coefficients
in the local system $p_1^*\PA{a}{-}\otimes p_2^*\PA{-}{b}$,
represented by the paths $\gamma_i$ together with horizontal sections $s_i$:
\begin{equation*}
c_i=\gamma_i\otimes s_i \in C_1(S^{[2]}; \,p_1^*\PA{a}{-}\otimes p_2^*\PA{-}{b}),
\qquad i=1,2,3,
\end{equation*}
Explicitly,
\begin{equation}
\begin{aligned}
c_1 &\colon \quad
\gamma_1(t)=(a, x(t)),
\qquad
&s_1&=1\otimes x\rvert_{[t, 1]}
\\
c_2 &\colon \quad
\gamma_2(t)=(x(t), b),
\qquad
&s_2&=x\rvert_{[0, t]}\otimes 1
\\
c_3 &\colon \quad
\gamma_3(t)=(x(t),\tau_{x(t)})
\qquad
&s_3&=x\rvert_{[0, t]}\otimes x\rvert_{[t, 1]}
\end{aligned}
\label{paths}
\end{equation}
The path $\gamma_3$ lies in the boundary component
corresponding to the diagonal.

One checks that the chain $c_1+c_2-c_3$ is a cycle and that its homology class is invariant under isotopies between paths. Thus, we obtain a map
\begin{equation}
c_1+c_2-c_{3}\colon\PA{a}{b}\to H_1(S^{[2]}; \,p_1^*\PA{a}{-}\otimes p_2^*\PA{-}{b})
\label{c1c2}
\end{equation}

Consider the homological long exact sequence associated
with the pair $(S^2, S^{[2]})$:
\begin{equation}
\begin{tikzcd}[cramped]
H_2(S^2;\, \mathcal{L}) \arrow[r]
& H_2(S^2, S^{[2]};\, \mathcal{L})\arrow[r]
& H_1(S^{[2]};\, \mathcal{L}) \arrow[r] 
& H_1(S^2;\, \mathcal{L})
\end{tikzcd}
\label{exact}
\end{equation}
where $\mathcal{L}$ denotes the local system $p_1^*\PA{a}{-}\otimes p_2^*\PA{-}{b}$ on 
$S^2$ and $S^{[2]}$.

By the K\"unneth formula, we have
\begin{equation*}
H_n(S^2;\, p_1^*\PA{a}{-}\otimes p_2^*\PA{-}{b})
=
\bigoplus_i
H_i(S;\,\PA{a}{-})
\otimes
H_{n-i}(S;\,\PA{-}{b}).
\end{equation*}
Since the local systems $\PA{a}{-}$ and $\PA{-}{b}$
are universal,
their homology vanishes in positive degrees.
Thus, the extreme terms of \eqref{exact} vanish.
Therefore, we obtain an isomorphism
\begin{equation}
H_1(S^{[2]};\, p_1^*\PA{a}{-}\otimes p_2^*\PA{-}{b})\cong H_2(S^2, S^{[2]};\, p_1^*\PA{a}{-}\otimes p_2^*\PA{-}{b}) 
\label{first}
\end{equation}
By the Thom isomorphism, we have
\begin{equation}
H_2(S^2, S^{[2]};\, p_1^*\PA{a}{-}\otimes p_2^*\PA{-}{b})\cong H_0(S;\, \PA{a}{-}\otimes \PA{-}{b})
\label{middle}
\end{equation}
By definition, 
\begin{equation}
H_0(S;\, \PA{a}{-}\otimes \PA{-}{b})=\PA{a}{-}\otimes_{\pi_1(S, -)} \PA{-}{b}=\PA{a}{b}
\label{last}
\end{equation}

Combining \eqref{c1c2} with the  chain of isomorphisms \eqref{first}--\eqref{last}
we obtain a map
\begin{equation}
\muop\colon \PA{a}{b}\to \PA{a}{b}
\label{mu}
\end{equation}

This construction may be reproduced word for word to obtain a pro-unipotent version of this map:
\begin{equation}
\muop\colon \PAu{a}{b}\to \PAu{a}{b}
\label{muu}
\end{equation}

\begin{rem}
The above construction is very close to, but does not coincide with,
the construction of \cite[Section~5]{HainTur}.
The main difference is that the latter construction
works for free loops rather than paths.
\end{rem}

The maps \eqref{mu} and \eqref{muu} admit a transparent group-theoretic interpretation.

Let
\begin{equation*}
K = \bigcap_{i=1,2}\ker p_{i*}
\subset \pi_1^{un}(S^{[2]})
\end{equation*}

\begin{prop}
\begin{enumerate}
\item\label{61}
For a surface $S$, points $a,b\in S$, and $K$ as above, one has an isomorphism
\begin{equation}
K/[K,K]\cong \PAu{a}{b}.
\label{KIso}
\end{equation}

\item\label{62}
The isomorphism \eqref{KIso} is given explicitly as follows:
the class $[x]$ is mapped to the class of the loop
$\gamma_1^{-1}X_{12}\gamma_1\in \fgu{S^{[2]}}{(a,b)}$:
\begin{equation}
\pau{S}{a}{b}\ni [x]
\longmapsto
[\gamma_1^{-1}X_{12}\gamma_1]\in K/[K,K],
\label{arrow}
\end{equation}
where $\gamma_1$ is defined by $x$ as in \eqref{paths}, and
\begin{equation*}
X_{12}\in
\pi_1^{un}(S^{[2]}, (b,\tau_b))
\end{equation*}
is represented by the loop in which one point moves once counterclockwise around the other.

\item\label{63}
For an immersion $x\colon [0,1]\to S$ as above,
consider the loop
\begin{equation}
\gamma_2\gamma_3^{-1}\gamma_1\in \pi^{un}_1(S^{[2]}, (a,b)),
\label{loop}
\end{equation}
where the paths $\gamma_i$ are defined in \eqref{paths}.
One verifies that the class of this loop belongs to  $K$. 
Under the identification \eqref{arrow}, the class of this loop in $K/[K,K]$ coincides with $\muop([x])$.
\end{enumerate}
\end{prop}

\begin{proof}
\eqref{61}
The chain of isomorphisms \eqref{first}--\eqref{last} identifies
$H_1(S^{[2]};\, p_1^*\PAu{a}{-}\otimes p_2^*\PAu{-}{b})$
with $\PAu{a}{b}$.
By arguments identical to those used in the proof of Proposition \ref{Hopent1},
we have
\begin{equation*}
H_1(S^{[2]};\, p_1^*\PAu{a}{-}\otimes p_2^*\PAu{-}{b})
\cong K/[K,K]
\end{equation*}

\eqref{62}
One directly checks that \eqref{arrow} defines the isomorphism \eqref{KIso}.

\eqref{63}
For any $T\in [0, 1]$ there is an element
\begin{equation}
({\gamma_1^{-1}}\rvert_{[0, 1-T]} {\gamma_2}\rvert_{[0, T]})
\, X_{12}\,
({\gamma_1}\rvert_{[T, 1]} {\gamma_2^{-1}}\rvert_{[1-T, 1]})
\in \fgu{S^{[2]}}{(a,b)},
\label{element}
\end{equation}
where 
$X_{12}\in
\pi_1^{un}(S^{[2]}, (x(T),\tau_{x(T)}))$
is represented by the loop in which one point moves once counterclockwise around the other.
All elements \eqref{element} represent the same class in $K/[K,K]$.
Writing the isomorphism \eqref{first} explicitly
(see the proof of Proposition \ref{Turaev} below),
one sees that $\muop([x])$ is equal to the sum of the classes represented by the elements \eqref{element} for values of $T$ corresponding to self-intersection points.
\end{proof}

The analogous statement may be formulated and proved in the same way in the non-pro-unipotent setting.

\begin{rem}
There is a remarkable similarity between the vanishing of $\muop$
and the homological pentagon equation as in Proposition \ref{Hopent1}.
In this sense, the conjecture that the vanishing of the self-intersection map of a path implies the triviality of loop \eqref{loop}
may be viewed as a toy model for the  conjecture mentioned in the introduction that regularized double shuffle relations imply the pentagon equation.
\end{rem}

In \cite{Turaev1979} (see also \cite{Massuyeau2018}) the 
notion of \emph{self-intersection map} was introduced. Here is an adaptation 
of this definition in our situation.

Let $x\colon [0, 1]\to S$ be an immersion of general position satisfying condition
\eqref{condition}. Let
\begin{equation}
D=\{(t_1, t_2)\in [0,1]^2\mid 0<t_1<t_2<1 \quad\mbox{and}\quad x(t_1)=x(t_2)\}
\label{D}
\end{equation}
For $d=(t_1, t_2)\in D$ denote by $y_d$ the element of $\pa{S}{a}{b}$
represented by 
\begin{equation*}
x\rvert_{[0, t_1]}x\rvert_{[t_2, 1]}
\end{equation*}
Let $\varepsilon_d\in\{\pm1\}$ denote the orientation sign of the basis
$(x'(t_1),x'(t_2))$.
The self-intersection of the path
is an element of $\PA{a}{b}$ given by
\begin{equation}
\sum_{d\in D}\varepsilon_d y_d
\label{self}
\end{equation}

In \cite[Section 3.3]{Massuyeau2018} it was shown that this map is continuous with respect
to the topology given by powers of the augmentation ideal, providing a  pro-unipotent
version of this definition.

\begin{prop}
The self-intersection map defined by \eqref{self}
coincides with the map $\muop$, see \eqref{mu}.
\label{Turaev}
\end{prop}

\begin{proof}
For an immersion $x\colon [0, 1]\to S$ as above
consider the 2-chain in $S^{[2]}$ with coefficients
in the local system $p_1^*\PA{a}{-}\otimes p_2^*\PA{-}{b}$,
represented by the submanifold 
\begin{equation*}
\{(x(t_1), x(t_2))\mid  0<t_1<t_2<1 \quad\mbox{and}\quad x(t_1)\neq x(t_2)\}
\end{equation*}
equipped with the horizontal section $x\rvert_{[0, t_1]}\otimes x\rvert_{[t_2, 1]}$.
The boundary of this submanifold consists of the union of 1-chains \eqref{paths}
and of boundary circles around the diagonal corresponding to elements of the set $D$, see
\eqref{D}. Thus, this 2-submanifold defines a homology between the
cycle $c_1+c_2-c_3$ defining  $\muop$ and the union of these circles. 
One checks that the homology class of each such circle corresponds to  $\varepsilon_d y_d$ 
under the isomorphism \eqref{last}.
\end{proof}

\begin{rem}
Our definition of $\muop$ differs slightly from definitions of \cite{Turaev1979} and 
\cite{Massuyeau2018}. In the first reference the self-intersection is defined for any path 
without condition \eqref{condition}, but it takes value in the quotient
of $\PA{a}{b}$ by the initial path. 
In the second paper, the self-intersection map is defined for arbitrary lifts of the path to the bundle of unit circles in the tangent bundle of $S$.
\end{rem}




\bibliographystyle{alpha}
\bibliography{dt.bib}

\end{document}